\documentclass[12pt]{amsart}
\usepackage{verbatim}
 
\usepackage{amssymb}
\usepackage{enumerate}

\usepackage{amsmath}
\usepackage{lineno}

\newtheorem{theorem}{Theorem}[section]
\newtheorem{proposition}[theorem]{Proposition}
\newtheorem{lemma}[theorem]{Lemma}
\newtheorem{cclaim}[theorem]{Claim}
\newtheorem{problem}[theorem]{Problem}
\newtheorem{ccase}{Case}
\newtheorem{scase}{Case}[ccase]
\newtheorem{sscase}{Case}[scase]

\newtheorem{obs}[theorem]{Observation}

\theoremstyle{definition}
\newtheorem{example}[theorem]{Example}{}

\newtheorem{definition}[theorem]{Definition}

\newcommand{\ccal}{{\mathcal C}}
\newcommand{\dcal}{{\mathcal D}}

\newcommand{\ical}{{\mathcal I}}

\newcommand{\pcal}{{\mathcal P}}

\newcommand{\setm}{\setminus}
\newcommand{\empt}{\emptyset}
\newcommand{\subs}{\subset}

\newcommand{\oo}{{\kappa}}
\newcommand{\oot}{{{\kappa}^+}}
\newcommand{\ooth}{{{\kappa}^{++}}}
\newcommand{\om}{{<{\kappa}}}

\def\<{\left\langle}
\def\>{\right\rangle}
\def\cf{\operatorname{cf}}

\def\br#1;#2;{\bigl[ {#1} \bigr]^ {#2} }

\def\to{\longrightarrow}
\def\rank{\operatorname{rk}}

\newcommand{\newcases}{\setcounter{ccase}{0}}

\newcommand{\pin}{\operatorname{\pi}}

\newcommand{\lev}[2]{\operatorname{I}_{#1}(#2)}

\newcommand{\htt}{\operatorname{ht}}

\newcommand{\aseq}[1]{\ccal(#1)}
\newcommand{\laseq}[2]{\ccal_{#1}(#2)}
\newcommand{\SEQ}{\operatorname{SEQ}}
\newcommand{\conseq}[2]{\<#1\>_{#2}}

\newcommand{\pib}{\operatorname{\pi}_B}
\newcommand{\piz}{\operatorname{\pi}}
\newcommand{\pim}{\operatorname{\pi}_-}
\newcommand{\bott}[1]{{#1}^-}
\newcommand{\topp}[1]{{#1}^+}
\newcommand{\incof}[1]{\operatorname{E}({#1})}

\newcommand{\eps}[2]{{\epsilon}^{#1}_{#2}}
\newcommand{\veps}[2]{{\varepsilon}^{#1}_{#2}}
\newcommand{\intpart}[1]{\operatorname{\mathcal E}({#1})}
\newcommand{\nn}[1]{\operatorname{n}({#1})}
\newcommand{\jjj}{\operatorname{j}}

\newcommand{\contint}[2]{\operatorname{I}({#1},#2)}
\newcommand{\jint}[1]{{\operatorname{J}(#1)}}

\newcommand{\oorbf}{\operatorname{o}}
\newcommand{\oorb}[1]{\oorbf(#1)}
\newcommand{\orbf}{\operatorname{o^*}}
\newcommand{\orb}[1]{\orbf(#1)}
\newcommand{\eorbf}{\operatorname{\overline{o}}}
\newcommand{\eorb}[1]{\eorbf(#1)}
\newcommand{\borbf}{\operatorname{o_B}}
\newcommand{\borb}[1]{\borbf(#1)}

\newcommand{\blocks}{\mathbb B}
\newcommand{\und}{X}
\newcommand{\block}[1]{B_{#1}}

\newcommand{\ifu}{\operatorname{i}}
\newcommand{\undef}{\text{undef}}

\newcommand{\irtaaa}[2]{{#2}^{-1}\{#1\}}
\newcommand{\concat}{\mathop{{}^{\frown}\makebox[-3pt]{}}}
\newcommand{\climit}{{\delta}}

\newcommand{\xn}[1]{x_{{\nu},#1}}
\newcommand{\xm}[1]{x_{{\mu},#1}}

\newcommand{\dlaseq}[2]{\dcal_{#1}(#2)}

\newcommand{\ml}{D}
\newcommand{\mb}{M}

\newcommand{\gu}{\operatorname{{\gamma}}}
\newcommand{\gd}{\operatorname{\underline{\gamma}}}

 \newcommand{\prer}{\preceq_r}
 \newcommand{\prenr}{\prec_r}
\newcommand{\ar}{A_r}
\newcommand{\ir}{\ifu_r}
 \newcommand{\prepq}{\preceq_{p,q}}
\newcommand{\prenpq}{\prec_{p,q}}

\newcommand{\mip}{{\prep}}

\newcommand{\mir}{{\prer}}

\newcommand{\al}{\alpha}
\newcommand{\be}{\beta}
\newcommand{\prep}{\preceq_p}

\newcommand{\pre}[1]{\preceq^{R#1}}
\newcommand{\pren}[1]{\prec^{R#1}}

\newtheorem{acase}{Case}

\begin{document}

\title{Cardinal sequences of LCS spaces under GCH}

\author[J. C. Martinez]{Juan Carlos Martinez}
\address{Facultat de Matem\`atiques \\ Universitat de Barcelona \\ Gran
  Via 585 \\ 08007 Barcelona, Spain}
\thanks{The first author was supported by
the Spanish Ministry of Education DGI grant MTM2005-00203}
\email{jcmartinez@ub.edu}
\author[L. Soukup]{
Lajos Soukup }
\thanks{The second author was partially supported by Hungarian National Foundation for Scientific Research grant no 61600. 
The research was started when the second 	author visited  
	the  Barcelona University.   
	The second  author would like to thank 
	Joan Bagaria and Juan-Carlos Mart\'\i nez for the arrangement of the 
	visit and their hospitality during the stay in Barcelona.
 }
\address{Alfr{\'e}d R{\'e}nyi Institute of Mathematics }
\email{soukup@renyi.hu}


\begin{abstract}

Let $\aseq {\alpha}$ denote the class of all cardinal sequences of
length $\alpha$ associated with compact scattered spaces.
Also put $$\laseq
{\lambda}{\alpha}=\{f\in \aseq{\alpha}: f(0)={\lambda} = \min[
f({\beta}) : \beta < {\alpha}]\}.$$

If ${\lambda}$ is a cardinal and ${\alpha}<{\lambda}^{++}$ is an
ordinal, we define $\dlaseq {\lambda}{\alpha}$ as follows: if
${\lambda}={\omega}$,
\begin{displaymath}
\dlaseq {\omega}{\alpha}=\{f\in
{}^{\alpha}\{{\omega},{\omega}_1\}: f(0)={\omega}\},  
\end{displaymath}
and if ${\lambda}$ is uncountable, 
\begin{multline}\notag
\dlaseq {\lambda}{{\alpha}}=\{f\in
{}^{\alpha}\{{\lambda},{\lambda}^+\}: f(0)={\lambda},\\
\irtaaa{{\lambda}}f\text{ is $<{\lambda}$-closed and successor-closed in
${\alpha}$}
\}.
\end{multline}
We show that for each uncountable regular cardinal ${\lambda}$ and ordinal
${\alpha}<{\lambda}^{++}$ it is consistent with GCH that
$\laseq {\lambda}{\alpha}$ is as large as possible, i.e.
$$
\laseq {\lambda}{\alpha}=\dlaseq {\lambda}{\alpha}.
$$
This yields  that  under GCH for  any sequence $f$ of regular cardinals of length ${\alpha}$
the following statements are equivalent:
\begin{enumerate}[(1)]
\item $f\in \aseq {\alpha}$ in some cardinal preserving and GCH-preserving
 generic-extension of the ground model.
\item  for some natural number $n$
there are infinite regular cardinals
$\lambda_0>\lambda_1>\dots>\lambda_{n-1}$  and ordinals
${\alpha}_0,\dots, {\alpha}_{n-1}$ such that
${\alpha}={\alpha}_0+\cdots+{\alpha}_{n-1}$ and $f=f_0\concat \
f_1\concat \cdots \ \concat f_{n-1}$ where each
$f_i\in\dlaseq{\lambda_i}{\alpha_i}$.
\end{enumerate}

The proofs are based on constructions of {\em universal}
locally compact scattered spaces.
\end{abstract}


\maketitle


\section{Introduction}

Given a locally compact scattered  $T_2$ (in short :
LCS) space  $X$
 the ${\alpha}^{\text{th}}$ Cantor-Bendixson level will be denoted by
$\lev{\alpha}X$.  The {\em height of $X$, $\htt(X)$,}
is the least ordinal ${\alpha}$ with $\lev {\alpha}X=\empt$.
The {\em reduced
height} $\htt^-(X)$ is the smallest ordinal $\alpha$ such that $\lev
{\alpha}{X}$ is finite.
Clearly, one has $\htt^-(X) \le \htt(X) \le \htt^-(X)+1.$
The {\em cardinal sequence} of  $X$,
denoted by $\SEQ(X)$,  is the
sequence of cardinalities of the
infinite Cantor-Bendixson levels of $X$, i.e.
\begin{displaymath}
\SEQ(X)=\bigl\langle\ |I_{\alpha}(X)|:{\alpha}<\htt(X)^-\ \bigr\rangle.
\end{displaymath}

A characterization in ZFC of the sequences of cardinals of
 length $\leq \omega_1$ that arise as cardinal sequences of LCS
 spaces is proved in \cite{JW2}. However, no characterization in ZFC is
 known for cardinal sequences of length $< \omega_2$.

For an ordinal ${\alpha}$ we let $\aseq {\alpha}$ denote the class of all cardinal sequences
of length $\alpha$ of LCS spaces.
We also put, for any fixed infinite cardinal $\lambda$,
\begin{displaymath}
\laseq {\lambda}{\alpha}=\{s\in \aseq{\alpha}: s(0)={\lambda}\land
\forall {\beta}<{\alpha}\ [s({\beta})\ge \lambda] \}.
\end{displaymath}

In \cite{JSW}, the authors show that a class $\aseq {\alpha}$ is
characterized if the classes $\laseq {\lambda}{\beta}$ are
characterized for every infinite cardinal $\lambda$ and every
ordinal $\beta\leq\alpha$. Then, they obtain under GCH a
characterization of the classes $\aseq {\alpha}$ for any ordinal
$\alpha <\omega_2$ by means of a  a full description under GCH of
the classes $\laseq {\lambda}{\alpha}$ for any ordinal
${\alpha}<{\omega}_2$ and any infinite cardinal ${\lambda}$. The
situation becomes, however, more complicated when we consider the
class $\aseq {\omega_2}$ . We can characterize under GCH the
classes $\laseq {\lambda}{\omega_2}$ for $\lambda
>\omega_1$, by using the description given in \cite{JSW} and the
following simple observation.

\begin{obs}\label{obs:union}
If ${\lambda}\ge {\omega}_2$, then
$f\in \laseq {\lambda}{{\omega}_2}$ iff
$f\restriction {\alpha}\in \laseq {\lambda}{\alpha}$ for each
${\alpha}<{\omega}_2$.
\end{obs}

\begin{proof}
If $\SEQ(X_{\alpha})=f\restriction {\alpha}$ for ${\alpha}<{\omega}_2$
then take $X$ as the disjoint union of $\{X_{\alpha}:{\alpha}<{\omega}_2\}$.
Then $\SEQ(X)=f$ because for any ${\beta}<{\omega}_2$ we have 
$\lev {\beta}X=
\bigcup\{\lev {\beta}{X_{\alpha}}:{\beta}<{\alpha}<{\omega}_2\}$
and so
\begin{equation}\notag
|\lev {\beta}X|=
\sum_{{\beta}<{\alpha}<{\omega}_2}|\lev {\beta}{X_{\alpha}}|=
{\omega}_2\cdot f({\beta})=f({\beta}).
\end{equation}
\end{proof}

If
${\alpha}$ is any ordinal, a subset $L\subs {\alpha}$ is called {\em
${\kappa}$-closed in ${\alpha}$}, where ${\kappa}$ is an infinite
cardinal, iff $\sup\<{\alpha}_i:i<{\kappa}\>\in L\cup\{{\alpha}\}$
for each increasing sequence $\<{\alpha}_i:i<{\kappa}\>\in
{}^{\kappa}L$.
The set $L$ is {\em
$<{\lambda}$-closed in ${\alpha}$} provided it is ${\kappa}$-closed in
${\alpha}$ for each cardinal ${\kappa}<{\lambda}$.
We say  that  $L$ is  {\em successor closed in
${\alpha}$} if ${\beta}+1\in L\cup\{{\alpha}\}$ for all ${\beta}\in
L$.

For a cardinal ${\lambda}$  and ordinal ${\delta}<{\lambda}^{++}$ 
we define $\dlaseq {\lambda}{\delta}$ as follows:
if
${\lambda}={\omega}$,
\begin{displaymath}
\dlaseq {\omega}{\delta}=\{f\in
{}^{\delta}\{{\omega},{\omega}_1\}: f(0)={\omega}\},  
\end{displaymath}
and if ${\lambda}$ is uncountable, 
\begin{multline}\notag
\dlaseq {\lambda}{\delta}=\{s\in
{}^{\delta}\{{\lambda},{\lambda}^+\}: s(0)={\lambda},\\
\irtaaa{{\lambda}}s\text{ is $<{\lambda}$-closed and successor-closed in
${\delta}$}
\}.
\end{multline}

The observation \ref{obs:union}  above left open the characterization of
$\laseq {{\omega}_1}{{\omega}_2}$ under GCH. In
\cite[Theorem 4.1]{JSW} it was proved that
 if
GCH holds  then
  \begin{equation}\notag
\laseq {{\omega}_1}{{\delta}}
\subseteq  \dlaseq {{\omega}_1}{{\delta}},
    \end{equation}
and  we have equality
for ${\delta}<{\omega}_2$.
In Theorem  \ref{tm:maingch} we show that it is consistent with GCH that
we have equality not only for ${\delta}={\omega}_2$ but
even for each ${\delta}<{\omega}_3$.

To formulate our results  we need to introduce some more notation.

We shall use the notation $\conseq {\kappa}{\alpha}$ to denote the
constant ${\kappa}$-valued sequence of length ${\alpha}$.
Let us denote the concatenation of a sequence $f$ of length $\alpha$ and a
sequence $g$ of length $\beta$ by $f \concat g$ so that the domain of
$f\concat g$ is $\alpha+\beta$ and $f \concat g(\xi)=f(\xi)$ for $\xi<\alpha$
and $f \concat g(\alpha+\xi)=g(\xi)$ for $\xi<\beta$.

\begin{definition}
An LCS space $X$ is called {\em $\laseq {\lambda}{\alpha}$-universal} iff
$\SEQ(X)\in \laseq {\lambda}{\alpha}$
and for each sequence $s\in \laseq {\lambda}{\alpha}$
there is an open subspace $Y$ of $X$ with
$\SEQ(Y)=s$.
\end{definition}

In this paper we prove the following result:
\begin{theorem}\label{tm:maingch}
If  ${\kappa}$ is an uncountable  regular cardinal with
${\kappa}^{<{\kappa}}={\kappa}$ and $2^{\kappa}={\kappa}^+$ then
 for each ${\delta}<{\kappa}^{++}$ there is
a ${\kappa}$-complete ${\kappa}^+$-c.c poset $P$ of cardinality
${\kappa}^+$ such that
in $V^P$
  \begin{equation}\notag
\laseq {{\kappa}}{{\delta}}
=\dlaseq {{\kappa}}{{\delta}}
  \end{equation}
and there is a
$\laseq {\oo}{{\delta}}$-universal LCS space.
\end{theorem}

How do the universal spaces come into the picture?
The first  idea to prove the consistency of
$\laseq {\lambda}{\alpha}=\dlaseq {\lambda}{\alpha}$ is to try to carry out an iterated forcing. For each $f\in \dlaseq {\lambda}{\alpha}$
we can try to find a poset $P_f$ such that
\begin{displaymath}
1_{P_f}\Vdash \text{There is an LCS space $X_f$ with
cardinal sequence $f$.}
\end{displaymath}
Since typically $|X_f|={\lambda}^+$, if we want to preserve the cardinals
and $CGH$ we should try to find a ${\lambda}$-complete, ${\lambda}^+$-c.c.
poset $P_f$ of cardinality ${\lambda}^+$.
In this case forcing with $P_f$ introduces
${\lambda}^+$ new subsets of ${\lambda}$  because $P_f$ has cardinality
${\lambda}^+$. However
$|\dlaseq {\lambda}{\alpha}|={\lambda}^{++}$! So the length of the iteration is at least ${\lambda}^{++}$, hence in the final model the cardinal
${\lambda}$ will have ${\lambda}^+\cdot {\lambda}^{++}={\lambda}^{++}$ many
new subsets, i.e. $2^{\lambda}>{\lambda}^+$.

A $\laseq {\lambda}{\delta}$-universal space has cardinality ${\lambda}^+$
so we  may hope that there is  a ${\lambda}$-complete, ${\lambda}^+$-c.c.
poset $P$ of cardinality ${\lambda}^+$ such  that
$V^P$ contains a $\laseq {\lambda}{\delta}$-universal space.
In this case $(2^{\lambda})^{V^P}\le
((|P|^{{\lambda}})^{\lambda})^{V}={\lambda}^+$.
So in the generic extension we might have $GCH$.

In this paper, we shall use the notion of a universal LCS space in order to
prove Theorem \ref{tm:maingch}.  Further constructions of universal LCS spaces will be
carried out in \cite{MS}.

\begin{problem}
Assume that $s$ is a sequence of cardinals of length ${\alpha}$,
$s\notin \aseq {\alpha}$. Is it possible that there  is a
$|{\alpha}|^+$-Baire ($|{\alpha}|^+$-complete) poset $P$ such that
$s\in \aseq {\alpha}$ in $V^P$?
\end{problem}

\newcommand{\lt}{{\mathcal L}_{\oo}}
\newcommand{\ltd}{\lt^{\delta}}
\noindent
For an ordinal ${\delta}<\ooth$
let $\ltd=\{{\alpha}<{\delta}:\cf({\alpha})\in \{\oo,\oot\}\}$.

\begin{definition}
An LCS space $X$ is called $\ltd$-good  iff
$X$ has a partition $X=Y\cup^*\bigcup^*\{Y_{\zeta}:{\zeta}\in \ltd\}$
such that
\begin{enumerate}[(1)]
\item $Y$ is an open subspace of $X$, $\SEQ(Y)=\conseq {\oo}{\delta}$,
\item $Y\cup Y_{\zeta}$ is an open subspace of $X$ with
 $\SEQ(Y\cup Y_{\zeta})=\conseq {\oo}{\zeta}\concat
\conseq {\oot}{\delta-{\zeta}}$.
\end{enumerate}
\end{definition}

Theorem \ref{tm:maingch} follows immediately from
Theorem \ref{tm:maingch2} and Proposition \ref{pr:ltd} below.

\begin{theorem}\label{tm:maingch2}
If  $\oo$ is an uncountable  regular cardinal with
${\oo}^{<{\oo}}={\oo}$
then for each ${\delta}<\ooth$ there is
a $\oo$-complete $\oot$-c.c poset $\pcal$ of cardinality $\oot$ such that
in $V^\pcal$
there is an $\ltd$-good space.
\end{theorem}

\begin{proposition}\label{pr:ltd}
Let ${\kappa}$ be an uncountable regular cardinal, 
${\delta}<{\kappa}^{++}$ and  $X$ be an $\ltd$-good space. 
Then for each $s\in \dlaseq{\oo}{\delta}$ 
there is an open subspace $Z$ of $X$
with $\SEQ(Z)=s$. Especially, under GCH an $\ltd$-good space
is $\laseq {\oo}{\delta}$-universal.
\end{proposition}

\begin{proof}
Let  $J=\irtaaa{\oot}s\cap \ltd$.
For each ${\zeta}\in J$ let
\begin{equation}\notag
f({\zeta})=\min (({\delta}+1)\setm (\irtaaa{\oot}s\cup {\zeta})).
\end{equation}
Let
\begin{equation}\notag
Z=Y\cup\bigcup\{\lev {<f({\zeta})}{Y\cup Y_{\zeta}}:{\zeta}\in J\}.
\end{equation}
Since $Y\cup Y_{\zeta}$ is an open subspace of $X$ it follows that
$\lev {<f({\zeta})}{Y\cup Y_{\zeta}}$ is an open subspace of $Z$.
Hence for every ${\alpha}<{\delta}$
\begin{multline}
\lev {\alpha}Z=\lev {\alpha}Y\cup
\bigcup\{\lev {\alpha}{\lev {<f({\zeta})}{Y\cup Y_{\zeta}}}:
{\zeta}\in J\}\\=
\lev {\alpha}Y\cup
\bigcup\{\lev {\alpha}{Y\cup Y_{\zeta}}:
{\zeta}\in J, {\zeta}\le {\alpha}<f({\zeta})\}.
\end{multline}
Since $[{\zeta}, f({\zeta}))\subs \irtaaa{\oot}s$ for ${\zeta}\in J$
it follows that if $s({\alpha})=\oo$ then
$\lev {\alpha}Z=\lev {\alpha}Y$, and so 
\begin{equation}
|\lev {\alpha}Z|= |\lev {\alpha}Y|=\oo.
\end{equation}
If  $s({\alpha})={\oot}$, 
let ${\zeta}_{\alpha}=\min\{{\zeta}\le
{\alpha}:[{\zeta},{\alpha}]\subs s^{-1}\{\oot\}\}$. Then 
${\zeta}_{\alpha}\in J$ because
$s(0)={\kappa}$ and 
$\irtaaa{\oo}s$
is $<\oo$-closed and successor-closed in
${\delta}$.
Thus ${\zeta}_{\alpha}\le {\alpha}<f({\zeta}_{\alpha})$ and so 
\begin{equation}
|\lev {\alpha}Z|\ge |\lev {\alpha}{Y\cup Y_{\zeta_{\alpha}}}|=\oot.
\end{equation}
Since $|Z|\le |X|=\oot$ we have $|\lev {\alpha}Z|=\oot$.
Thus $\SEQ(Z)=s$.
\end{proof}

Theorem \ref{tm:maingch} yields the following  characterization:

\begin{theorem}
Under GCH for  any sequence $f$ of regular cardinals of length ${\alpha}$
the following statements are equivalent:
\begin{enumerate}[(A)]
\item $f\in \aseq {\alpha}$ in some cardinal preserving and GCH-preserving
 generic-extension of the ground model.
\item  for some natural number $n$
there are infinite regular cardinals
$\lambda_0>\lambda_1>\dots>\lambda_{n-1}$  and ordinals
${\alpha}_0,\dots, {\alpha}_{n-1}$ such that
${\alpha}={\alpha}_0+\cdots+{\alpha}_{n-1}$ and $f=f_0\concat \
f_1\concat \cdots \ \concat f_{n-1}$ where each
$f_i\in\dlaseq{\lambda_i}{\alpha_i}$.
\end{enumerate}
\end{theorem}

\begin{proof}
 (A) clearly implies (B) by \cite{JSW}.

Assume now that (B) holds. Without loss of
generality, we may suppose that  $\lambda_{n-1} = \omega$. Since
the notion of forcing defined in Theorem \ref{tm:maingch} 
preserves GCH, we can carry out a cardinal-preserving and
GCH-preserving 
iterated
forcing of length $n-1$, $\<P_m:m < n-1\>$,  such that for $m <
n-1$

\begin{equation}\notag
V^{P_m}\models
\laseq{\lambda_m}{\alpha_m}=\dlaseq{\lambda_m}{\alpha_m}.
\end{equation}
Put $k=n-2$, $\beta = {\alpha}_0+\cdots+{\alpha}_k$ and
$g=f_0\concat f_1\concat \cdots \ \concat f_k$. Since $f_m\in
\dlaseq{\lambda_m}{\alpha_m}\cap V$, in $V^{P_k}$ we have $f_m\in
\laseq{\lambda_m}{\alpha_m}$ for each $m<n-1$. Hence in $V^{P_k}$
we have $g\in \aseq {\beta}$ by  \cite[Lemma 2.2]{JSW}. Also, by using
 \cite[Theorem 9]{JW2}, we infer that $f_{n-1}\in \aseq {\alpha_{n-1}}$ in
ZFC. Then as $f = g \concat f_{n-1}$, in $V^{P_k}$ we have  $f\in
\aseq {\alpha}$ again by  \cite[Lemma 2.2]{JSW}.
\end{proof}

\begin{problem}
{\rm \bf (1)} Are  (A) and (B) below equivalent under GCH 
for every sequence $f $of regular cardinals?
\begin{enumerate}[(1)]
\item[(A)] $f\in \aseq {\alpha}$.
\item[(B)]  for some natural number $n$
there are infinite regular cardinals
$\lambda_0>\lambda_1>\dots>\lambda_{n-1}$  and ordinals
${\alpha}_0,\dots ,{\alpha}_{n-1}$ such that
${\alpha}={\alpha}_0+\cdots+{\alpha}_{n-1}$ and $f=f_0\concat \
f_1\concat \cdots \ \concat f_{n-1}$ where each
$f_i\in\dlaseq{\lambda_i}{\alpha_i}$.
\end{enumerate}
{\rm\bf (2)} Is it consistent with GCH that (A) and (B) above are equivalent
for every sequence of regular cardinals?
\end{problem}

Juh\'asz and  Weiss proved in \cite{JW} that
$\conseq {{\omega}}{\delta}\in \aseq {\delta}$ for each
${\delta}<{\omega}_2$.

Also, it was shown in \cite{M} that for every specific regular
cardinal $\kappa$ it is consistent that $\langle \kappa
\rangle_{\delta}\in \ccal (\delta)$ for each $\delta
<\kappa^{++}$. However, the following problem is open:

\begin{problem}
Is it consistent with $GCH$ that
$\conseq {{\omega}_1}{\delta}\in \aseq {\delta}$ for each
${\delta}<{\omega}_3$?
\end{problem}

\section{Proof of theorem \ref{tm:maingch2}}

This section is devoted to the proof of Theorem 1.6, so $\oo$ is an
uncountable regular cardinal with ${\oo}^{<{\oo}}={\oo}$, and ${\delta}<\ooth$
is an ordinal.

If ${\alpha}\le {\beta}$ are ordinals let
\begin{equation}
[{\alpha},{\beta})=\{{\gamma}:{\alpha}\le {\gamma}<{\beta}\}.
\end{equation}
We say that $I$ is an {\em ordinal interval} iff there are
ordinals ${\alpha}$ and ${\beta}$ with $I=[{\alpha},{\beta})$.
Write $\bott I={\alpha}$ and $\topp I={\beta}$.

If $I=[{\alpha},{\beta})$ is an ordinal interval let $\incof
I=\{\veps I{\nu}:{\nu}<\cf({\beta})\}$ be a cofinal closed subset
of $I$ having order type $\cf {\beta}$ with ${\alpha} = \veps
I{0}$ and put
\begin{equation}
\intpart I=\{[\veps I{\nu},\veps I{\nu+1}):{\nu}<\cf{\beta}\}
\end{equation}
 provided ${\beta}$ is a limit ordinal,
and let $\incof I =\{{\alpha},{\beta}'\}$
 and put
\begin{equation}
\intpart I=\{[{\alpha},{\beta}'),\{{\beta}'\}\}
  \end{equation}
provided ${\beta}={\beta}'+1$.

\vspace{1mm} Define $\{\ical_n:n<{\omega}\}$ as follows:
\begin{equation}
\ical_0=\{[0,{\delta})\} \text{ and }
\ical_{n+1}=\bigcup\{\intpart I:I\in \ical_n\}.
\end{equation}
Put $\mathbb I=\bigcup\{\ical_n:n<{\omega}\}$. Note that $\mathbb
I$ is a {\em cofinal tree of intervals} in the sense defined in \cite{M}.
Then, for each ${\alpha}<{\delta}$ we define
\begin{equation}
\nn {\alpha}=\min\{n:\exists I\in \ical_n \mbox{ with } \bott
I={\alpha}\},
\end{equation}
and for each  ${\alpha}<{\delta}$ and $n<{\omega}$ we define
\begin{equation}
\contint {\alpha}n\in \ical_n \text{ such that } {\alpha}\in
\contint {\alpha}n.
\end{equation}

\begin{proposition}
\label{Proposition-2.1} Assume that $\zeta <
\delta$ is a limit ordinal.  Then, there is a $j(\zeta) \in
\omega$ and an interval $J({\zeta})\in \ical_{j(\zeta)}$ such that
$\zeta$ is a limit point of $E(J(\zeta))$. Also, we have $\nn
{\zeta}-1\le j({\zeta})\le\nn {\zeta}$, and $j(\zeta) = n(\zeta)$
if $cf(\zeta) = \oot$.
\end{proposition}

\begin{proof}
Clearly $j(\zeta)$ and $\jint {\zeta}$ are unique if defined.

If there is an $I \in \ical_{\nn {\zeta}}$ with $\topp I={\zeta}$
then $J(\zeta)=I$, and so $j(\zeta)=\nn {\zeta}$. If there is no
such $I$, then ${\zeta}$ is a limit point of $\incof{\contint
{\zeta}{\nn {\zeta}-1}}$, so $J(\zeta)= \contint {\zeta}{\nn
{\zeta}-1}$ and $j(\zeta)=\nn {\zeta}-1$.

Assume now that $\cf({\zeta})=\oot$. Then
 ${\zeta}\in \incof{\contint {{\zeta}}{\nn {\zeta}-1}}$,
but $|\incof{\contint {{\zeta}}{\nn {\zeta}-1}}\cap {\zeta}|\le
\oo$, so ${\zeta}$ can not be a limit point of $\incof{\contint
{{\zeta}}{\nn {\zeta}-1}}$. Therefore, it has a predecessor
${\xi}$ in $\incof{\contint {{\zeta}}{\nn {\zeta}-1}}$, i.e
$[{\xi},{\zeta}) \in \ical_{\nn {\zeta}}$, and so
$J(\zeta)=[{\xi},{\zeta})$ and
 $j(\zeta)=\nn {\zeta}$.
\end{proof}

\begin{example}
\label{Example-2.2}
Put $\delta
= \omega_2\cdot \omega_2 + 1$. We define

\vspace{1mm} $E([0,\delta)) = \{0,{\omega}_2\cdot {\omega}_2\}$,

\vspace{1mm} $E([0,{\omega}_2\cdot {\omega}_2)) =
\{{\omega}_2\cdot \xi:0\leq \xi <{\omega}_2\}$,

\vspace{1mm} $E([{\omega}_2\cdot \xi,{\omega}_2\cdot (\xi + 1))) =
\{\zeta: {\omega}_2\cdot \xi \leq \zeta < {\omega}_2\cdot (\xi +
1)\}$,

\vspace{1mm} $E(\{\zeta\}) = \{\zeta\}$ for each $\zeta \leq
{\omega}_2\cdot {\omega}_2$.

\vspace{1mm} Then, we have $\nn{{\omega}_2\cdot {\omega}_2} = 1$,
$\nn{{\omega}_2\cdot {\omega}_1} = 2$, $\nn{{\omega}_2\cdot
{\omega}_1 + {\omega} } = 3$. Also, we have $j({\omega}_2\cdot
{\omega}_2) = \jjj({\omega}_2\cdot {\omega}_1) = 1$ and
$J({\omega}_2\cdot {\omega}_2) = J({\omega}_2\cdot {\omega}_1) =
[0,{\omega}_2\cdot {\omega}_2)$.
\end{example}

\vspace{2mm} If $\cf(J(\zeta)^+)\in \{\oo,\oot\}$, we denote by
$\{\eps {\zeta}{\nu}:{\nu}<\cf(\topp {J(\zeta)})\}$  the
increasing enumeration of $\incof {J(\zeta)}$, i.e.  
$\eps {\zeta}{\nu}=\veps {J({\zeta})}{\nu}$  for 
${\nu}<\cf (\topp {J(\zeta)})$.

\vspace{2mm} Now if ${\zeta}<{\delta}$, we define the {\em basic
orbit} of $\zeta$ (with respect to $\mathbb I$) as
\begin{equation}
\oorb {\zeta}=\bigcup\{(\incof{\contint{\zeta}m}\cap {\zeta}):
m<\nn {\zeta}\}.
\end{equation}

\vspace{1mm} Note that this is the notion of orbit used in \cite{M} in
order to construct by forcing an LCS space $X$ such that $\SEQ(X)
= \langle \kappa \rangle_{\eta}$ for any specific regular cardinal
$\kappa$ and any ordinal $\eta < \kappa^{++}$. However, this
notion of orbit can not be used to construct an LCS space $X$ such
that $\SEQ(X) = \langle \kappa \rangle_{\kappa^+}\concat \;\langle
\kappa^+ \rangle$. To check this point, assume on the contrary
that such a space $X$ can be constructed by forcing from the
notion of a basic orbit. Then, since the basic orbit of $\kappa^+$
is $\{0\}$, we have that if $x,y$ are any two different elements
of $I_{\kappa^+}(X)$ and $U,V$ are basic neighbourhoods of $x,y$
respectively, then $U\cap V\subset I_0(X)$. But then, we deduce
that $|I_1(X)| = \kappa^+$.

\vspace{1mm} However, we will show that a refinement of the notion
of basic orbit can be used to proof Theorem 1.6.

\vspace{1mm} If  ${\zeta}<{\delta}$  with  $\cf {\zeta}\ge \oo$,
we define the {\em extended orbit} of $\zeta$ by
\begin{equation}
\eorb {\zeta}= \oorb {\zeta}\cup (\incof {J(\zeta)}\cap {\zeta}).
\end{equation}

\vspace{1mm} Consider the tree of intervals defined in Example-2.2. 
Then, we have $\oorb{{\omega}_2\cdot {\omega}_1} =
\overline{o}({\omega}_2\cdot {\omega}_1) = \{{\omega}_2\cdot
\xi:0\leq \xi < {\omega}_1\}$, $o({\omega}_2\cdot {\omega}_2) =
\{0\}$, $\overline{o}({\omega}_2\cdot {\omega}_2) =
\{{\omega}_2\cdot \xi:0\leq \xi <{\omega}_2\}$.

\vspace{1mm} Note that if $\zeta < \delta$, the basic orbit of
$\zeta$ is a set of cardinality at most $\kappa$ (see
\cite[Proposition 1.3]{M}). Then, it is easy to see that for any $\zeta
< \delta$ with $\cf {\zeta}\ge \oo$ , the extended orbit of
$\zeta$ is a cofinal subset of ${\zeta}$ of cardinality $\cf
{\zeta}$.

\vspace{1mm} In order to define the desired notion of forcing, we
need some preparations. The underlying set of the desired space
will be the union of a collection of blocks.

Let
\begin{equation}
\blocks=\{S\}\cup\{\<{\zeta},{\eta}\>:{\zeta}<{\delta},\cf
{\zeta}\in \{\oo,\oot\}, {\eta}<\oot\}.
\end{equation}
Let
\begin{equation}
\block S={\delta}\times \oo
\end{equation}
and
\begin{equation}
\block {{\zeta},{\eta}}=\{\<{\zeta},{\eta}\>\}\times
       [{\zeta},{\delta})\times
\oo
\end{equation}
for $\<{\zeta},{\eta}\>\in \blocks\setm \{S\}$.

Let
\begin{equation}
\und=\bigcup\{\block T:T\in \blocks\}.
\end{equation}

The underlying set of our space will be $\und$. We should produce
a partition $X=Y\cup^*\bigcup^*\{Y_{\zeta}:{\zeta}\in \ltd\}$ such
that
\begin{enumerate}[(1)]
\item $Y$ is an open subspace of $X$ with $\SEQ(Y)=\conseq
{\oo}{\delta},$ \item $Y\cup Y_{\zeta}$ is an open subspace of $X$
with
 $\SEQ(Y\cup Y_{\zeta})=\conseq {\oo}{\zeta}\concat
\conseq {\oot}{\delta-{\zeta}}$.
\end{enumerate}
We will have  $Y=\block S$, $Y_{\zeta}=\bigcup\{\block
{{\zeta},{\eta}}:{\eta}<\oot\}$ for ${\zeta}\in \ltd$.

Let
\begin{equation}
\piz:\und\to {\delta}  \text{ such that  }
\begin{array}{l}
\piz (\<{\alpha},{\nu}\>)={\alpha},\\
\piz (\<{{\zeta},{\eta},\alpha},{\nu}\>)={\alpha}.
\end{array}
\end{equation}

Let
\begin{equation}
\pim:\und\to {\delta}  \text{ such that  }
\begin{array}{l}
\pim (\<{\alpha},{\nu}\>)={\alpha},\\
\pim (\<{{\zeta},{\eta},\alpha},{\nu}\>)={\zeta}.
\end{array}
\end{equation}

Define
\begin{equation}
\pib:\und\to \blocks  \text{ by the formula } x\in \block {\pib
(x)}.
\end{equation}

Define the {\em block orbit} function $\borbf:\blocks\setm
\{S\}\to \br {\delta};\le \oo;$ as follows:
\begin{equation}
\borb {\<{\zeta},{\eta}\>}=\left\{
\begin{array}{ll}
\eorb{\zeta}&\text{if $\cf {\zeta}=\oo$},\\
\oorb {\zeta}\cup \{\eps {\zeta}{\nu}:{\nu}<{\eta}\}& \text{if
$\cf {\zeta}=\oot$}.\end{array} \right.
\end{equation}

That is, if $\cf{\zeta}=\oot$ then
\begin{equation}\notag
\borb {\<{\zeta},{\eta}\>}=\eorb {\zeta}\cap \eps {\zeta} {\eta}.
\end{equation}

Finally we define the {\em orbits} of the  elements  of $\und$ as
follows:
\begin{equation}
\orbf:\und \to \br {\delta};\le \oo; \text{ such that  }
\begin{array}{l}
\orb  {\<{\alpha},{\nu}\>}=\oorb {\alpha},\\
\orb {\<{{\zeta},{\eta},\alpha},{\nu}\>}=\borb
     {\<{\zeta},{\eta}\>}\cup (\oorb {{\alpha}}\setm {\zeta}).
\end{array}
\end{equation}

Let $\Lambda\in \mathbb I$ and $\{x,y\}\in \br \und;2;$. We say
that {\em $\Lambda$ isolates $x$ from $y$} if
\begin{enumerate}[(i)]
\item  $\bott \Lambda <\piz (x)<\topp \Lambda$, \item $\topp
\Lambda\le \piz (y) $ provided $\pib(x)=\pib(y)$, \item $\topp
\Lambda\le  \pi_{-}(y) $ provided $\pib(x)\ne\pib(y)$.
\end{enumerate}

\vspace{2mm} Now, we define the poset $\pcal=\<P,\le\>$ as
follows: $\<A,\preceq,i\>\in P$ iff
\begin{enumerate}[(P1)]
\item $A\in \br \und;\om;$. \item $\preceq$ is a partial order on
$A$ such that
$x\preceq y$ implies $x=y$ or $\piz(x)<\piz (y)$. 
\item Let $x\preceq y$.
  \begin{enumerate}[(a)]
    \item If $\pib(y)=\<{\zeta},{\eta}\>$ and ${\zeta}\le \piz(x)$
    then $\pib(x)=\pib(y)$.
    \item If $\pib(y)=\<{\zeta},{\eta}\>$ and ${\zeta}> \piz(x)$ then
    $\pib(x)=S$.
    \item If $\pib(y)=S$ then   $\pib(x)=S$.
  \end{enumerate}
\item $\ifu:\br A;2; \to A\cup\{\undef\}$ such that
for each $\{x,y\}\in \br A;2;$ we have
  \begin{equation}\notag
  \forall a\in A ([a\preceq x\land a\preceq y]\text{ iff } a\preceq \ifu\{x,y\}).
  \end{equation}
\item $\forall \{x,y\}\in \br A;2;$ if $x$ and $y$ are
$\preceq$-incomparable but $\preceq$-compatible, then
$\piz(\ifu\{x,y\})\in \orb x\cap \orb y$.

\item Let $\{x,y\}\in [A]^2$ with $x\preceq y$. Then:
\begin{enumerate}[(a)]
    \item If $\pi_B(x)=S$ and $\Lambda\in \mathbb I$ isolates $x$
    from $y$, then there is
 $z\in A$ such that $x\preceq z\preceq y$ and $\piz
      (z)=\topp {\Lambda}$.
      \item If $\pi_B(x)\neq S$, $\pi(x)\neq \pi_{-}(x)$ and $\Lambda\in \mathbb I$ isolates $x$
    from $y$, then there is
 $z\in A$ such that $x\preceq z\preceq y$ and $\piz
      (z)=\topp {\Lambda}$.
      \end{enumerate}

\end{enumerate}

The ordering on $P$ is the extension: $\<A,\preceq,\ifu\>\le
\<A',\preceq',\ifu'\>$
 iff $A'\subset A$, $\preceq'=\preceq\cap (A'\times A')$, and
$\ifu'\subs \ifu$.

\vspace{2mm} By using (P3), we obtain:

\begin{cclaim}
\label{Claim-2.1}  Assume that $x,y,z$ and
$\Lambda$ are as in (P6). Then we have:

(a) If $\pi_B(x)=\pi_B(y)$, then $\pi_B(z)=\pi_B(x)=\pi_B(y)$.

(b) If $\pi_B(x)\neq\pi_B(y)$ and $\Lambda^+ < \pi_{-}(y)$, then
$\pi_B(z)=\pi_B(x)$.

(c) If $\pi_B(x)\neq\pi_B(y)$ and $\Lambda^+ = \pi_{-}(y)$, then
$\pi_B(z)=\pi_B(y)$.
\end{cclaim}

  Since $\kappa^{<\kappa}=\kappa$ implies
  $(\kappa^+)^{<\kappa}=\kappa^+$, 
  we have
  that the cardinality of $P$ is $\kappa^+$. 
Then,  using the arguments of \cite{M} it is enough to prove
that Lemmas \ref{Lemma-2.2}, \ref{Lemma-2.3} and \ref{Lemma-2.4}  below hold.

\begin{lemma}
\label{Lemma-2.2} $\pcal$ is
$\oo$-complete.
\end{lemma}

\begin{lemma}
\label{Lemma-2.3}  $\pcal$
 satisfies the $\oot$-c.c.
\end{lemma}

\begin{lemma}
\label{Lemma-2.4} 
 Assume that $p=\<A,\preceq,\ifu\>\in  P$, $x\in A$,
and ${\alpha}<\piz (x)$. Then  there is $p'=\<A',\preceq',\ifu'\>\in P$
with $p'\le p$ and there is $b\in A'\setminus A$ with  $\piz
 (b)={\alpha}$ 
such that  $b \preceq' y$ iff  $x\preceq
y$     for $y\in A$.
\end{lemma}

\vspace{2mm} Since ${\kappa}$ is regular,  Lemma \ref{Lemma-2.2} clearly
holds.

\begin{proof}[PROOF of Lemma \ref{Lemma-2.4}]
Let ${\beta}=\piz (x)$. Let $K$ be a countable subset of $
[{\alpha},{\beta})$ such that  ${\alpha}\in K$ and $\topp
{\contint {\gamma}n }\in K\cup [{\beta},{\delta})$ for
${\gamma}\in K$ and $n<{\omega}$.  For each ${\gamma}\in K$ pick
$b_{\gamma}\in X\setm A$ such that $\piz(b_{\gamma})={\gamma}$ and
\begin{enumerate}[(1)]
\item if $\pib(x)=S$ then $\pib(b_{\gamma})=S$. \item if
$\pib(x)\ne S$ and ${\gamma}\ge\pim(x)$ then $\pib(
b_{\gamma})=\pib(x)$. \item if $\pib(x)\ne S$ and
${\gamma}<\pim(x)$ then $\pib( b_{\gamma})=S$.
\end{enumerate}

Let $A'=A\cup\{b_{\gamma}:{\gamma}\in K\}$,
\begin{multline}\notag
\preceq'=\preceq\cup\{\<b_{\gamma},b_{{\gamma}'}\>:
{\gamma},{\gamma}'\in K, {\gamma}\le {\gamma}'\}
\\\cup\{\<b_{\gamma},z\>:{\gamma}\in K, z\in A, x\preceq z\}.
\end{multline}
The definition of $\ifu'$ is straightforward because if $y\in A'$
and ${\gamma}\in K$ then either $y$ and $b_{\gamma}$ are
$\preceq'$-comparable or they are $\preceq'$-incompatible.

Then $p'=\<A',\preceq',i'\>$ and $b=b_{\alpha}$ satisfy the requirements.
\end{proof}

Finally we should prove Lemma \ref{Lemma-2.3}.

\begin{proof}[Proof of Lemma \ref{Lemma-2.3}]

Assume that $\<r_{\nu}:{\nu}<\oot\>\subs P$ with $r_{\nu}\neq r_{\mu}$ for $\nu <\mu<\kappa^+$.

Write $r_{\nu}=\<A_{\nu},\preceq_{\nu},\ifu_{\nu}\>$ and
$A_{\nu}=\{\xn i:i<{\sigma}_{\nu}\}$.

Since we are assuming that $\kappa^{<\kappa}=\kappa$, by thinning
out $\<r_{\nu}:{\nu}<\oot\>$ by means of standard combinatorial
arguments, we can assume the following:
\begin{enumerate}[(A)]
\item ${\sigma}_{\nu}={\sigma}$ for each ${\nu}<\oot$. \item
$\{A_{\nu}:{\nu}<\oot\}$ forms a $\Delta$-system with kernel $A$.
\item For each ${\nu}<{\mu}<\oot$ there is an isomorphism
$h=h_{{\nu},{\mu}}:\<A_{\nu},\preceq_{\nu}, \ifu_{\nu}\>\to
\<A_{\mu},\preceq_{\mu},\ifu_{\mu}\>$ such that 
for every $i<\sigma$ and $x,y\in A_{\nu}$ the following holds:
\begin{enumerate}[(a)]
  \item $h\restriction A=\operatorname{id}.$
  \item $h(\xn i)=\xm i.$
\item $\pib(x)=\pib(y)$  iff $\pib(h(x))=\pib(h(y))$. \item
$\pib(x)=S$ iff $\pib(h(x))=S.$\item $\pi(x)=\pi_{-}(x)$ iff
$\pi(h(x))=\pi_{-}(h(x))$.   
\item if $\{x,y\}\in \br A;2;$ then 
  $\ifu_{\nu}\{x,y\}=\ifu_{\mu}\{x,y\}$.
  
\end{enumerate} 

\end{enumerate}

\vspace{2mm}
  Note that in order to obtain (C)(f) we use condition (P5) and the fact that $|o^*(x)|\leq \kappa$ for every $x\in A$. Also, we may assume the following:

\begin{enumerate}[(A)]
\addtocounter{enumi}{3}
\item There is a partition ${\sigma}=K\cup^* F\cup^* L\cup^*
\ml\cup^* \mb$ such that for each ${\nu}<{\mu}<\oot$:
\begin{enumerate}[(a)]
\item $\forall i\in K$ $\xn i\in A$ and so $\xn i=\xm i$. $A=\{\xn
i:i\in K\}$. \item $\forall i\in F$ $\xn i\ne \xm i$ but $\pib(\xn
i)=\pib (\xm i)\ne S$. \item $\forall i\in L$ $\pib (\xn i)\ne
\pib (\xm i)$ but $\pim (\xn i)= \pim (\xm i)$. \item  $\forall
i\in \ml$
 $\pib (\xn i)=S$ and $\piz(\xn i)\ne \piz(\xm i)$.
\item $\forall i\in \mb$ $\pib (\xn i)\ne S$ and $\pim (\xn i)\ne
\pim (\xm i)$.
\end{enumerate}
\item If $\pib ({\xn i})=\pib ({\xn j})$ then $\{i,j\}\in \br K\cup
  \ml;2;\cup \br K\cup F;2;\cup
 \br L;2; \cup \br \mb;2;$.

\end{enumerate}

It is well-known that if   ${\gamma}<\oo=\oo^{<\oo}$ then the
following partition relation holds:
\begin{equation}
\notag \oot\to (\oot, ({\omega})_{\gamma})^2.
\end{equation}
Hence we can assume:
\begin{enumerate}[(A)]
\addtocounter{enumi}{5}
\item If ${\nu}<{\mu}<\oot$ then for each $i\in {\sigma}$ we have
\begin{enumerate}[(a)]
\item \label{g} $\piz(\xn i)\le \piz(\xm i)$, \item \label{h}
$\pim(\xn i)\le \pim(\xm i)$.
\end{enumerate}
\end{enumerate}

By (F)(a) and (F)(b)  the sequences $\{\piz (\xn i):{\nu}<\oot \}$
and $\{\pim (\xn i):{\nu}<\oot \}$ are increasing for each $i\in
{\sigma}$, hence the following definition is meaningful:

For $i \in {\sigma}$ let
\begin{equation}
 \notag
{\climit}_i=\left\{
\begin{array}{ll}
\piz (\xn i)&\text{if $i\in K$},\\
\sup\{\piz (\xn i):{\nu}<\oot \}&\text{if $i\in F\cup \ml$},\\
\pim (\xn i)&\text{if $i\in L$},\\
\sup\{\pim (\xn i):{\nu}<\oot \}&\text{if $i\in \mb$}.
\end{array}
\right.
\end{equation}

By using Proposition \ref{Proposition-2.1}, (C)(c) and condition (P3), we obtain:

\begin{cclaim}
\label{Claim-2.5}

 (a) If $i\in F\cup
D\cup M$, then $cf(\delta_i) = \kappa^+$ and   $\mbox{
sup}(J(\delta_i))=\delta_i$.   Moreover 
for
every $\nu < \kappa^+$  we have  
$\pi(x_{\nu,i}) < \delta_i$ if $i\in F\cup D$, and 
$\pim(x_{\nu,i}) < \delta_i$ 
 if $i\in M$.

(b) If $\{i,j\}\in [L]^2\cup [M]^2$ and $x_{\nu,i} \prec_{\nu}
x_{\nu,j}$ for $\nu <\kappa^+$, then $\delta_i = \delta_j$.

\end{cclaim}

Indeed, (b) holds for large enough ${\nu}$, and so (C)(c) implies that
  it holds for each ${\nu}$.

 We put
\begin{equation}
Z_0 = \{\pim(\xn i):i\in F\cup K, \pib (\xn i)\ne
S\}\cup \{{\climit}_i:i\in {\sigma}\}.
\end{equation}

Since $\pi''A=\{{\delta}_i:i\in K\}$ we have $\pi''A\subs Z_0$.
Then, we define $Z$ as the closure of $Z_0$ with respect
to $\mathbb I$:

\begin{equation}
  Z = Z_0\cup \{I^+: I\in \mathbb I,I\cap Z_0\neq
  \emptyset\}.
  \end{equation}

  \vspace{4mm}

Since $|Z| <\kappa$, we can assume:

\begin{enumerate}[(A)]
\addtocounter{enumi}{6}
\item\label{az}  $\; A=\{\xn i: i\in K\cup F\cup \ml, \piz (\xn i)\in Z\}.$
\end{enumerate}

Equivalently,
\begin{equation}
\text{if $i\in F\cup \ml$ then $\piz (\xn i)\notin Z$}.
\end{equation}

Let us remark that for $i\in L\cup M$ we may have  
that $\pin(\xn i)\in Z$. 

\medskip

 Our aim is to show that there are $\nu<\mu<\kappa^+$ such that
$r_{\nu}$ and $r_{\mu}$ are compatible. Note that if $x,y\in A$
with $x\neq y$ then, by (C)(f), we may assure that
$i_{\nu}\{x,y\}=i_{\mu}\{x,y\}$. However, if $x\in
A_{\nu}\setminus A$ and $y\in A_{\mu}\setminus A$ it may happen
that for infinitely many $v\in A$ we have $v\preceq_{\nu}x$ and
$v\preceq_{\mu}y$. Then, in order to amalgamate $r_{\nu}$ and
$r_{\mu}$ in such a way that any pair of such elements has an
infimum in the amalgamation, we will need to add new elements to
$A_{\nu}\cup A_{\mu}$. Then, the next definitions will permit us to find
suitable room for adding new elements to the domains of the
conditions.

Let

    $$\sigma_1 = \{i\in \sigma\setminus K: \cf(\delta_i)=\kappa
    \}$$

   \noindent  and

    $$\sigma_2 = \{i\in \sigma\setminus K: \cf(\delta_i)=\kappa^+
    \}.$$
    
\vspace{2mm}
    \noindent Assume that $i\in \sigma\setminus K$. Put $I_i=J(\delta_i)$.
    Let

    $$\xi_i = \mbox{ min}\{\nu \in \cf \delta_i: \epsilon^{I_i}_{\nu}
    > \mbox{ sup}(\delta_i\cap Z)\}.$$

    \noindent Then, if $i\in \sigma_1$ we put

    $$\underline{\gamma}(\delta_i) = \epsilon^{I_i}_{\xi_i} \mbox{
    and } \gamma(\delta_i) = \delta_i,$$

    \noindent and if $i\in \sigma_2$ we put

    $$\underline{\gamma}(\delta_i) = \epsilon^{I_i}_{\xi_i} \mbox{
    and } \gamma(\delta_i) = \epsilon^{I_i}_{\xi_i + \kappa}.$$

\begin{cclaim}
\label{Claim-2.6} For each $i\in F\cup D
\cup M$ there is ${\nu}_i<\oot$ such that for all $ \nu_i\leq \nu
<\kappa^+$ we have:

\begin{equation}
 \mbox{if } i\in F\cup D  \mbox{ then } \pi(x_{\nu,i})\in
 J(\delta_i)\setminus \gamma(\delta_i)
\end{equation}
and
\begin{equation}
\mbox{if } i\in  M \mbox{ then } \pi_{-}(x_{\nu,i})\in
 J(\delta_i)\setminus \gamma(\delta_i).
\end{equation}
\end{cclaim}

\begin{proof}
For $i \in F\cup D\cup M$ we have

\begin{equation}
{\climit}_i=\left\{
\begin{array}{ll}
\sup\{\piz (\xn i):{\nu}<\oot \}&\text{if $i\in F\cup \ml$},\\
\sup\{\pim (\xn i):{\nu}<\oot \}&\text{if $i\in \mb$},
\end{array}
\right.
\end{equation}
and $\gamma(\delta_i)<\sup (\jint {{\climit}_i})=\climit_i$.
\end{proof}

\begin{cclaim}
\label{Claim-2.7} For each $i\in L$ with
$\mbox{ cf}(\delta_i)=\kappa^+$ there is ${\nu}_i<\oot$ such that
for all $ \nu_i\leq \nu <\kappa^+$, $o^*(x_{\nu,i})\supset
\overline{o}(\delta_i)\cap \gamma(\delta_i).$
\end{cclaim}

\begin{definition}
\label{Definition 2.8} $r_{\nu}$ is {\em
good} iff
  \begin{enumerate}[(i)]
    \item $\forall i\in F\cup \ml $ ${\piz(\xn i)}\in\jint{{\climit}_i}
\setm \gamma(\delta_i)$.
    \item $\forall i\in \mb $ $\pim(\xn i)\in J(\delta_i) \setm \gamma(\delta_i) $.
     \item $\forall i\in L$ if $\cf {\climit}_i=\oot$ then $\orb {\xn
    i}\supset \overline{o}(\delta_i)\cap \gamma(\delta_i).$
  \end{enumerate}
\end{definition}

Using  Claims \ref{Claim-2.6} and \ref{Claim-2.7} we can assume:

\begin{enumerate}[(A)]
\addtocounter{enumi}{7} 
\item \label{good}$r_{\nu}$ is
good  for  ${\nu}<\oot$.
\end{enumerate}

By using $(\ref{good})$, we will prove that $r_{\nu}$ and $r_{\mu}$ are
compatible for $\{\nu,\mu\}\in [\kappa^+]^2$. First, we need to
prove some fundamental facts.

By using (P3), (E) and (C)(c)  we obtain:

\begin{cclaim}
\label{Claim-2.9} If $\xn i
\preceq_{\nu} \xn j$ then either $\pib (\xn i)=S$ or $\pib (\xn
i)=\pib(\xn j)$ and $\{i,j\}\in [K\cup F]^2\cup [L]^2\cup [M]^2$.
\end{cclaim}

Indeed, (P3) and  (E) imply that Claim \ref{Claim-2.9} holds for large enough
  ${\nu}$,
and then (C)(c) yields that it holds for each ${\nu}$.

\begin{cclaim}
\label{Claim-2.10}  If $\xn i
\preceq_{\nu} \xn j$ then ${\climit}_i\le{\climit}_j$.
\end{cclaim}

\begin{proof}
If $\xn i\preceq_{\nu} \xn j$ then $\xm i\preceq_{\mu} \xm j$ for
each ${\mu}<\oot$, and so we have:
\begin{enumerate}[(a)]
\item  $\piz (\xm i)\le \piz (\xm j)$, 
\item $\pim (\xm i)\le \pim
(\xm j)$, \item  if 
{$\pib (\xm i)\ne \pi_B(x_{\mu,j})$}
then $\piz
(\xm i)\le \pim (\xm j)$.
\end{enumerate}
Hence if $\pib (\xn i)\ne \pib (\xn j)$ then
\begin{equation}
\climit_i= \sup \{\piz (\xm i):{\mu}<\oot\} \le \sup \{\pim (\xm
j):{\mu}<\oot\}\le \climit_j.
\end{equation}
If $\pib (\xn i)= \pib (\xn j)$ then either $\{i,j\}\in \br K\cup
F;2;\cup \br K\cup \ml;2;$ and so
\begin{equation}
\climit_i= \sup \{\piz (\xm i):{\mu}<\oot\} \le \sup \{\piz (\xm
j):{\mu}<\oot\}= \climit_j,
\end{equation}
or $\{i,j\}\in \br L;2;\cup \br \mb;2;$ and so
\begin{equation}
\climit_i= \sup \{\pim (\xm i):{\mu}<\oot\} \le \sup \{\pim (\xm
j):{\mu}<\oot\} = \climit_j.
\end{equation}
\end{proof}

\begin{cclaim}
\label{Claim-2.11}  Assume $i,j\in
{\sigma}$. If $\xn i\preceq_{\nu} \xn j$ then either
${\climit}_i={\climit}_j$ or there is $a\in A$ with $\xn
i\preceq_{\nu} a\preceq_{\nu}\xn j $.
\end{cclaim}

\begin{proof}
Put $x_i=x_{\nu,i}, x_j=x_{\nu,j}$. Assume that $i,j\not\in K$ and
$\delta_i\neq \delta_j$. By Claim \ref{Claim-2.10}, we have $\delta_i
<\delta_j$. Since $i\in L\cup M$ implies $\delta_i =\delta_j$, we
have that $i\in F\cup D$, and so $\pi(x_i)<\delta_i$, $\mbox{
cf}(\delta_i)=\kappa^+$ and $J(\delta_i)^+=\delta_i$. We
distinguish the following cases:

\vspace{2mm}\noindent {\bf Case 1}. $i\in D$ and $j\in D\cup L\cup
M$.

Since $\delta_i <\delta_j$, we have that $J(\delta_i)$ isolates
$x_i$ from $x_j$. Also, note that if $j\in L\cup M$, then
$J(\delta_i)^+ = \delta_i < \pi_{-}(x_j)$. By (P6)(a), we infer
that there is an $x=x_{\nu,k}\in A_{\nu}$ such that
$\pi(x)=\delta_i$ and $x_i\prec_{\nu} x \prec_{\nu} x_j$. Now, by
Claim \ref{Claim-2.1}(a)-(b), we deduce that $k\in K\cup D$. But as
$\delta_i\in Z$, by (\ref{az}), we have that $x\in A$, and so we are
done.

\vspace{2mm}\noindent {\bf Case 2}. $i\in D$ and $j\in F$.

We have that $\pi_B(x_i)\neq \pi_B(x_j)$. By using (P3), we infer
that $\delta_i\leq \pi_{-}(x_j)$, and so $J(\delta_i)$ isolates
$x_i$ from $x_j$. If $\delta_i < \pi_{-}(x_j)$, we proceed as in
Case 1. So, assume that $\delta_i = \pi_{-}(x_j)$. By (P6)(a), we
deduce that there is an $x=x_{\nu,k}\in A_{\nu}$ such that
$\pi(x)=\delta_i$ and $x_i\prec_{\nu} x \prec_{\nu} x_j$. By Claim
\ref{Claim-2.1}(c), we infer that $k\in K\cup F$. Then  as $\delta_i\in Z$, we
have that $x\in A$ by (\ref{az}).

\vspace{2mm}\noindent {\bf Case 3}. $i,j\in F$.

We have that $\pi_B(x_i)=\pi_B(x_j)\neq S$ and $J(\delta_i)$
isolates $x_i$ from $x_j$. Since $\pi_{-}(x_i)\in Z$ and we are
assuming that $i\not\in K$, we infer that $\pi(x_i)\neq
\pi_{-}(x_i)$. Now, applying (P6)(b), we deduce that there is an
$x=x_{\nu,k}\in A_{\nu}$ such that $\pi(x)=\delta_i$ and
$x_i\prec_{\nu} x \prec_{\nu} x_j$. Now we deduce from Claim \ref{Claim-2.1}(a)
that $k\in K\cup F$. Then  as  $\delta_i\in Z$, we have that $x\in
A$ by (\ref{az}).
\end{proof}

\begin{cclaim}
\label{Claim-2.12} 
If $x\in A$ and 
$y\in
A_{\nu}$, and $x$ and $y$ are compatible but incomparable in
$r_{\nu}$, then $\ifu_{\nu}\{x,y\}\in A$.
\end{cclaim}

\begin{proof}
Indeed, $\piz(\ifu_{\nu}{\{x,y\}})\in \orb x$ by (P5) and $|\orb x|\le
\oo$.
\end{proof}

\newcommand{\preq}{\preceq_q}
\newcommand{\aq}{A_q}
\newcommand{\ap}{A_p}
\newcommand{\ip}{\ifu_p}
\newcommand{\iq}{\ifu_q}

\newcommand{\xp}[1]{x^p_{#1}}
\newcommand{\xq}[1]{x^q_{#1}}

\begin{cclaim}
\label{Claim-2.13} Assume that  $\xn i $
and $ \xn j$ are compatible but incomparable in $r_{\nu}$. Let
$\xn k=\ifu_{\nu}\{\xn i,\xn j\}$. Then either $\xn k\in A$ or
${\climit}_i={\climit}_j={\climit}_k$.
\end{cclaim}

\begin{proof}
Assume $x_{\nu,k}\not\in A$. Then $k\not\in
K$. If 
${\climit}_k\ne {\climit}_i$, we infer that there is $b\in A$
with $x_{\nu,k}\preceq_{\nu} b \preceq_{\nu} x_{\nu,i}$ by Claim 
\ref{Claim-2.11}..
 So
$\xn k=\ifu_{\nu}\{b,\xn j\}$ and thus $\xn k\in A$ by Claim \ref{Claim-2.12},
contradiction.

Thus ${\climit}_i={\climit}_k$, and similarly
${\climit}_j={\climit}_k$.
\end{proof}

After this preparation fix $\{{\nu},{\mu}\}\in \br \oot;2;$. We do
not assume that ${\nu}<{\mu}$! Let $p=r_{\nu}$ and $q=r_{\mu}$.
Our purpose is to show that $p$ and $q$ are compatible. Write
$p=\<\ap,\prep,\ip\>$ and $q=\<\aq,\preq,\iq\>$, $\xp i=\xn i$ and
$\xq i=\xm i$, $\climit_{\xp i}=\climit_{\xq i}=\climit_i$.

If $s=\xp i$ write $s\in K$ iff $i\in K$. Define $s\in L$, $s\in
F$, $s\in \mb$, $s\in \ml$ similarly.

In order to amalgamate conditions $p$ and $q$, we will use a
refinement of the notion of amalgamation given in \cite[Definition
2.4]{M}. 

 Let
$A'=\{\xp i:i\in F\cup D\cup M\cup L\}$.

Let $\rank:\<A',\prep\restriction A'\>\to \theta$ be an
order-preserving injective function for some ordinal
$\theta<\oo$. 

For $x\in A'$, by induction on $\rank(x)<\theta$ choose
${\beta}_{x}\in {\delta}$ as follows:

Assume that $\rank(x)={\tau}$ and ${\beta}_{z}$ is defined
provided
 $\rank(z)<{\tau}$.

Let
\begin{equation}
{\beta}_x=\min \bigl ( \bigl(\eorb{{\climit}_x}\cap
[\underline{\gamma}(\delta_x),\gamma(\delta_x))) \setm \sup\{
  {\beta}_{z}: z\prec_p x
\}\bigr ).
\end{equation}

Since $z\prep  x$ implies $\climit_z\le \climit_x$ by Claim \ref{Claim-2.10},
we have ${\beta}_z <\gamma(\delta_x)$ for $z\prec_p x$. Since
{$\cf(\gamma(\delta_x))= \oo$}
and $|A'|<{\kappa}$ we have
$\sup\{{\beta}_{z}: z\prec_p x \} <\gamma(\delta_x)$, so
${\beta}_x$ is always defined.

For $x\in A'$ let
\begin{equation}
y_x=\left\{
\begin{array}{ll}
\<{\beta}_x,\rank(x)\>&\text{if $x\in L\cup \ml\cup \mb$},\\
\<{\zeta},{\eta},{\beta}_x,\rank (x)\>&\text{if $x\in F$, $\pib
({x})=\<{\zeta},{\eta}\>$.}
\end{array}
\right.
\end{equation}

Put
\begin{equation}
Y=\{y_x:x\in A'\}.
\end{equation}

\newcommand{\gbar}{\bar g}

For $x\in A'$ put
\begin{equation}
\text{ $g(y_x)=x$ and $\gbar(y_x)=x'$},
\end{equation}
where $x'$ is the ``twin'' of $x$ in $\aq$ (i.e.
$h_{\nu,\mu}(x)=x'$).

\vspace{2mm} We will include the elements of $Y$ in the domain of
the amalgamation $r$ of $p$ and $q$. In this way, we will be able
to define the infimum in $r$ of elements $s,t$ where $s\in
A_p\setminus A_q$ and $t\in A_q\setminus A_p$.

\vspace{2mm} We need to prove some basic facts.

\begin{cclaim}
\label{Claim-2.14}  If $x\in A'$ then

$$ \overline{o}(\delta_x)\cap
[\underline{\gamma}(\delta_x),\gamma(\delta_x))\subset o^*(x)\cap
o^*(x').$$
\end{cclaim}

\begin{proof} Let $\al\in \overline{o}(\delta_x)\cap
[\underline{\gamma}(\delta_x),\gamma(\delta_x)).$  It is enough to
show that $\al\in o^*(x)$. Note that if $x\in D$, then $\al\in
o(\pi(x))=o^*(x)$. If $x\in M$, we have that $\al\in
o(\pi_{-}(x))\subset o_B(\pi_B(x)) \subset o^*(x)$. Also, if $x\in
L$ then as $p$ is good we have that $\al\in o_B(\pi_B(x)) \subset
o^*(x)$. Now, assume that $x\in F$. Since $\pi_{-}(x)\in Z$, we
have that $\pi_{-}(x) < \underline{\gamma}(\delta_x)$, hence
$\al\in o(\pi(x))\setminus \pi_{-}(x)$, and so $\al\in o^*(x)$.
\end{proof}

Note that we obtain as an immediate consequence of Claim \ref{Claim-2.14} that
$\beta_x\in o^*(x)\cap o^*(x')$ for every $x\in A'$.

\begin{cclaim}
\label{Claim-2.15} If $x \in A'$ then
\begin{equation}
\orb {y_x} \supset  (\orb {x}\cap \piz (y_x))\cup \{{\beta}_z:
{\climit}_z={\climit}_x\land z \prec_p x\}.
\end{equation}
\end{cclaim}

\begin{proof}  Note that if $I\in \mathbb I$ and $\alpha,\beta\in E(I)$ with
$\alpha <\beta$, we have that $\alpha\in o(\beta)$. By using this
fact, it is easy to verify that $\{\be_z: \delta_z = \delta_x$ and
$z \prec_p x\} \subset o^*(y_x)$.

 Now we prove that $o^*(y_x)\supset o^*(x)\cap \pi(y_x)$.
Suppose that $\zeta\in o^*(x)\cap \pi(y_x)$. We distinguish the
following three cases:

\vspace{2mm}\noindent {\bf Case 1}. $x\in D$.

 Then 
{$x,y_x\in B_S$}, and so we have $o^*(x) = o(\pi(x))$ and $o^*(y_x) =
o(\pi(y_x))=o(\beta_x)$. 
{Let $k=j({\delta}_x)$, i.e. $J(\delta_x)\in \ical_k$.}
Since $\zeta \in o(\pi(x))\cap
\pi(y_x)$, we infer that $\zeta \in E(I(\pi(x),m))\cap \pi(y_x)$
for some $m\leq k$. Note that for $m\leq k$ we have $I(\pi(x),m)=
I(\pi(y_x),m)$. So, $\zeta\in o(\pi(y_x))= o^*(y_x)$.

\vspace{2mm}\noindent {\bf Case 2}. $x\in L\cup M$.

Since $\zeta\in o^*(x)\cap \pi(y_x)$, we infer that $\zeta\in
o_B(\pi_B(x))$. Then as $y_x\in B_S$, we can show that $\zeta \in
o(\pi(y_x))=o^*(y_x)$ by using an argument similar to the one
given in Case 1.

\vspace{2mm}\noindent {\bf Case 3}. $x\in F$.

We have $\pi_B(x) = \pi_B(y_x)\neq S$. Put $(\xi,\eta)= \pi_B(x) =
\pi_B(y_x)$. So,

$o^*(x) = o_B((\xi,\eta))\cup (o(\pi(x))\setminus \pi_{-}(x))$,

$o^*(y_x) = o_B((\xi,\eta))\cup (o(\pi(y_x))\setminus
\pi_{-}(x))$.

So we may assume that $\zeta\in o(\pi(x))\setminus \pi_{-}(x)$,
and then we can proceed as in Case 1.
\end{proof}

\begin{cclaim}
\label{Claim-2.16} There are no  $y\in
Y$ and
 $a\in A$ such that $a\prep g(y),\overline{g}(y)$ and
$\pi(y) \leq \pi(a)$.
\end{cclaim}

\begin{proof}
Assume that $y\in Y$. Put $x=g(y)$ and $I=J(\delta_{x})$. Note
that if $x\in F\cup D\cup M$, then since $\mbox{ sup}(I\cap
Z)<\underline{\gamma}(\delta_{x})$  we infer that there is no
$a\in A$ such that $a \prep x$ and $\pi(a)\geq \pi(y)$.

Now, suppose that $x\in L$. Note that there is no $a\in A$ such
that $a \prec_p x$ and $\pi_B(a)= \pi_B(x)$. Also, as $\mbox{
sup}(\delta_{x}\cap Z)<\underline{\gamma}(\delta_{x})$, we infer
that there is no $a\in A\cap B_S$ such that $a\prep x$ and
$\pi(a)\geq \pi(y)$.
\end{proof}

\begin{cclaim}
\label{Claim-2.17} If $x\in F\cup D\cup
M$, then there is no interval that isolates $y_x$ from $x$.
\end{cclaim}

\begin{proof} By Claim \ref{Claim-2.5}(a), we have $\mbox{ cf}(\delta_x) =
\kappa^+$ and $\pi(x) <\delta_x$. By Proposition-2.1, we have
$j(\delta_x)=n(\delta_x)$ and $\delta_x=J(\delta_x)^+$. Then,
assume on the contrary that there is an interval $\Lambda\in
\mathbb I$ that isolates $y_x$ from $x$. Let $m<\omega$ such that
$\Lambda = I(\pi(y_x),m)$. As $\Lambda$ isolates $y_x$ from $x$
and $x,y_x\in J(\delta_x)$, we deduce that $m>j(\delta_x)$. But
from $m>j(\delta_x)$ and $\pi(y_x)\in E(J(\delta_x))$ we infer
that $\pi(y_x)=\Lambda^-$. Hence, $\Lambda$ does not isolate $y_x$
from $x$.
\end{proof}

However, if $x\in L$ it may happen that there is a $\Lambda\in
\mathbb I$ that isolates $y_x$ from $x$.

\vspace{2mm} Now, we are ready to start to define the common
extension $r=(A_r,\prec_r,i_r)$ of $p$ and $q$. First, we define
the universe $A_r$. Put 
{$L^+ = \{x\in L:
\pi(x)\neq\pi_{-}(x)\}$.} 
Then, if $x\in L^+$ and $x'$ is the twin
element of $x$, we consider new elements $u_x,u_{x'}\in X\setminus
(A_p\cup A_q\cup Y)$ such that $\pi_B(u_x)=\pi_B(x)$,
$\pi(u_x)=\pi_{-}(x)$, $\pi_B(u_{x'})=\pi_B(x')$  and
$\pi(u_{x'})=\pi_{-}(x')$. We suppose that $u_x,u_z,u_{x'},u_{z'}$
are different if $x,z$ are different elements of $L^+$. We put $U
= \{u_x: x\in L^+\}$ and $U' = \{u_{x'}: x\in L^+\}$. Then, we
define

$$A_r = A_p\cup A_q\cup Y\cup U\cup U'.$$

Clearly, $A_r$ satisfies (P1). Now, our purpose is to define
$\preceq_r$. First, for $x,y\in \br  \ap\cup \aq;2;$ let
\begin{equation}
x\prepq y \text{ iff } \exists z\in \ap\cup \aq\ [x\prep z\lor
x\preq z]\ \land\ [z\prep y\lor z\preq y].
\end{equation}
The following claim is straightforward.
\begin{cclaim}
\label{Claim-2.18}  $\prepq$ is the
partial order on $\ap\cup \aq$ generated by $\prep\cup \preq$.
\end{cclaim}

\vspace{2mm} Next, we define the relation $\preceq^*$ on $A_p\cup
A_q\cup Y$ as follows. Let us recall that $A=\ap\cap \aq$.
Informally, $\preceq^*$ will be  the ordering on $A_p\cup A_q\cup Y$
generated by
\begin{multline}\notag
\prepq\cup \{\<y,g(y)\>,\<y,\gbar(y)\>:y\in Y\}\cup
\\\{\<y,y'\>:y,y'\in Y, g(y)\prep g(y')\}\cup\\
\{\<a,y\>:a\in A, y\in Y, a\prep g(y)\}.
\end{multline}
The formal definition is a bit different, but its formulation
simplifies the separation of different cases later. So we introduce
five relations on $A_p\cup A_q\cup Y$ as follows:
\begin{equation}\notag
  \begin{array}{lcl}
\pren{1_{p}}&=&\{\<y,a\>:y\in Y,a\in \ap, g(y)\prep a\},\\
\pren{1_{q}}&=&\{\<y,a\>:y\in Y,a\in \aq, \gbar(y)\preq a\},\\
\pre{2}&=&\{\<y,y'\>:y,y'\in Y, g(y)\prep g(y')\},\\
\pren{3_{p}}&=&\{\<x,y\>:x\in \ap,y\in Y, \exists a\in A\
x\prep a \prep g(y)\},\\
\pren{3_{q}}&=&\{\<x,y\>:x\in \aq, y\in Y, \exists a\in A\ x\preq a
\preq \gbar(y)\}.
  \end{array}
\end{equation}
Then, we put
\begin{equation}
\preceq^*=\preceq_{p,q}\cup \pren{1_{p}}\cup \pren{1_{q}}\cup
\pre{2}\cup \pren{3_{p}}\cup \pren{3_{q}}.
\end{equation}

\vspace{2mm} The partial order $\prer$ will be an extension of
$\preceq^*$. So, we need to prove the following lemma:

\begin{lemma}
\label{Lemma-2.19} $\preceq^*$ is a partial
order on $A_p\cup A_q\cup Y$.
\end{lemma}

\begin{proof}
Let $s\prer t\prer u$. We should show that $s\prer u$.

We can assume that $t\notin \aq\setm \ap$.

\newcases

\renewcommand{\theccase}{\Roman{ccase}}

\renewcommand{\thesscase}{\thescase.\roman{sscase}}

\begin{ccase}
$s\in \ap\cup \aq$, $t\in \ap$ and $s\prepq t$.
\end{ccase}

Without loss of generality, we may assume that $u\in Y$ and
$t\pren{3p}  u$, i.e. there is $a\in A$ such that $t\prep a\prep
g(u)$.

\begin{scase}
$s  \in \ap$.
\end{scase}

Then  $s\prep a\prep g(u)$ and so $s\pren{3p}u$.

\begin{scase}
$s\in \aq\setm \ap$.
\end{scase}

Then there is $b\in A$ such that $s\preq b \prep t\prep  a \prep
g(u)$. Then $s \preq a\preq \gbar(u)$ so $s\pren{3q}u$.

 \begin{ccase}
$s\in Y$, $t\in \ap$ and
 $s\pren{1p} t$.
\end{ccase}

 \begin{scase} $u\in \ap\cup \aq$ and
$s\pren{1p} t\prepq u$.
\end{scase}

 \begin{sscase}
 $u\in \ap$.
\end{sscase}

Then $g(s)\prep t\prep u$ hence $s\pren{1p}u$.

 \begin{sscase}
 $u\in \aq\setm \ap$.
 \end{sscase}

Then there is $a\in A$ such that $g(s)\prep t\prep a\preq u$.
Hence $\gbar (s)\preq a \preq u$ and so $\gbar(s)\preq u$. Thus
$s\pren{1q}u$.

\begin{scase}$u\in Y$ and
$s\pren{1p} t\pren{3p} u$.
\end{scase}

Then there is $a\in A$ such that $g(s)\prep t \prep a \prep g(u)$
and so  $s\pre{2} u$.

 \begin{ccase}
$s,t\in Y$ and $s\pre{2}t$.
\end{ccase}

   \begin{scase} $u\in \ap$ and $s\pre {2}t\pren {1p} u$.
   \end{scase}

Then $g(s)\prep g(t)\prep u$ so $s\pren {1p}u$.

   \begin{scase} $u\in \aq$ and
    $s\pre{2}t\pren {1q} u$.
   \end{scase}

Then $g(s)\prep g(t)$ and $\gbar (t)\preq u$. Thus $\gbar(s)\preq
\gbar(t)\preq u$ so $s\pren {1q}u$.

\begin{scase}
$u\in Y$ and $s\pre{2}t \pre{2} u$.
\end{scase}

Then $g(s)\prep g(t)\prep g(u)$ so $s\pre{2} u$.

 \begin{ccase}
$s\in \ap$, $t\in Y$ and
 $s\pren{3p}t$.
 \end{ccase}

 \begin{scase}
$u\in \ap$ and $s\pren {3p}t\pren {1p}u$.
\end{scase}

Then there is $a\in A$ such that $s\prep a \prep g(t)\prep u$ so
$s\prep u$.

\begin{scase}
$u\in \aq$ and $s\pren {3p}t\pren {1q}u$.
 \end{scase}

Then there is $a\in A$ such that $s\prep a \prep g(t)$ and
$\gbar(t)\preq u$. So $a\preq \gbar(t)$ and hence $s\prep a\preq
u$. Thus  $s\prepq u$.

\begin{scase} $u\in Y$ and
$s\pren {3p}t\pre{2} u$.
\end{scase}

Then there is $a\in A$ such that $s\prep a \prep g(t)\prep g(u)$
and so $s\pren {3p}u$.

 \begin{ccase}
$s\in \aq$, $t\in Y$ and $s\pren {3q}t$.
 \end{ccase}

Only case (3) is different from (IV): \addtocounter{scase}{2}

\begin{scase} $u\in Y$ and $s\pren {3q}t\pre{2} u$.
\end{scase}

Then there is $a\in A$ such that $s\preq a \preq \gbar(t)$ and
$g(t) \prep g(u)$. Then $\gbar(t)\preq \gbar(u)$, so $s\preq
a\preq \gbar(u)$, thus $s\pren {3q}u$.
\end{proof}

Informally, $\prer$ will be  the ordering on $A_p\cup A_q\cup Y\cup
U\cup U'$
generated by
\begin{equation}
\notag
\preceq^*\cup 
 \{\<y_s,u_s\>: s\in A_p\cup A_q\}\cup
\{\<u_s,s\>: s\in A_p\cup A_q\}.
 \end{equation}

\vspace{2mm} Now, in order to define $\prer$ we need to make the
following definitions:

\begin{equation}\notag
  \begin{array}{lcl}
\pren{4_{p}}&=&\{\<s,u_x\>: s\in A_p\cup A_q\cup Y, x\in L^+ \mbox{ and } s\preceq^* y_x\},\\
\pren{4_{q}}&=&\{\<s,u_{x'}\>:  s\in A_p\cup A_q\cup Y, x\in L^+ \mbox{ and } s\preceq^* y_x \},\\
\pren{5_{p}}&=&\{\<u_x,t\>: x\in L^+, t\in A_p \mbox{ and } x\prep t\},\\
\pren{5_{q}}&=&\{\<u_{x'},t\>: x\in L^+, t\in A_q \mbox{ and }
x'\preceq_q t\},\\
=^U&=&\{\<u_x,u_x\>: x\in L^+\},\\
=^{U'}&=&\{\<u_{x'},u_{x'}\>: x\in L^+\}.  
\end{array}
\end{equation}

Then, we define:

\begin{equation}
\prer = \preceq^*\cup \pren{4_{p}}\cup \pren{4_{q}}\cup
\pren{5_{p}}\cup \pren{5_{q}}\cup =^U\cup =^{U'}.
\end{equation}

Write $x\prenr y$ iff $x\prer y$ and $x\ne y$.

\begin{lemma}
\label{Lemma-2.20} $\prer$ is a partial
order on $A_r$.
\end{lemma}

\begin{proof}
Assume that $s\prenr t\prenr v$. We have to show that $s\prenr v$. Note
that if $s,t,v\in A_p\cup A_q\cup Y$, then $s\prec^* t \prec^* v$, and so
we are done by Lemma \ref{Lemma-2.19}. Also, it is impossible that two
elements of $\{s,t,v\}$ are in $U\cup U'$. To check this point,
assume that $s,v\in U$. Put $s=u_x$, $v=u_z$ for $x,z\in  L^+$. As
$u_x\prec_r t$, we have $u_x \pren{5p} t$ and so $x\prep t$. As $t\prec_r
u_z$, we have $t \pren{4p} u_z$ and so $t\prec^* y_z$. Hence,
$x\prep t \prec^* y_z \prec^* z$. Since $x\;\mip\; t$ and $x\in L$, we
infer that $t\in L$. Also, from $t\prec^* y_z$ we deduce that
$t\pren{3p} y_z$ and so there is an $a\in A$ such that $t\;\mip\;
a\;\mip\; z$. But since $t\in L$, it is impossible that there is
an $a\in A$ with $t\;\mip\; a$. Proceeding in an analogous way, we
arrive to a contradiction if we assume that $s\in U$ and $v\in
U'$. So, at most one element of $\{s,t,v\}$ is in $U\cup U'$.
Then, we consider the following cases:

\vspace{2mm}\noindent {\bf Case 1}. $s\in U$.

We have that $t,v\in A_p\cup A_q\cup Y$. Put $s=u_x$ for some
$x\in L^+$. Since $u_x \prec_r t$, we have $u_x \pren{5p} t$ and so
$x\;\mip\; t$. As $t \prec_r v$, we have $t \prec^* v$. So, $x\prep
t\prec^* v$. But as $x\in L$ and $x\;\mip\; t$, we infer that $t\in
L$ . Hence, $t\prec_p v$. Thus $x\prec_p v$, therefore $u_x \pren{5p} v$,
and so $u_x \prec_r v$.

\vspace{2mm}\noindent {\bf Case 2}. $t\in U$.

We have that $s,v\in A_p\cup A_q\cup Y$. Put $t=u_x$ for $x\in
L^+$. From $s\prec_r u_x$, we infer that $s\pren{4p} u_x$ and so
$s\preceq^* y_x$. From $u_x \prec_r v$, we deduce that $u_x\pren{5p} v$ and
hence $x\;\mip\; v$. So we have $s\preceq^* y_x \prec^* x\;\mip\; v$, and
therefore $s\prec_r v$.

\vspace{2mm}\noindent {\bf Case 3}. $v\in U$.

We have that $s,t\in A_p\cup A_q\cup Y$. Put $v=u_x$ for $x\in
L^+$. Since $t\prec_r u_x$, we have that $t\pren{4p}u_x$ and so $t\preceq^*
y_x$. And from $s\prec_r t$ we deduce that $s\prec^*t$. So $s\prec^* y_x$,
hence $s\pren{4p} u_x$, and thus $s\prec_r u_x$.
\end{proof}

Now note that $s \pren{3_p} t$ implies $\pin(s) <\pin(t)$ by Claim
\ref{Claim-2.16}, and so it is clear that $s \prec_r t$ 
implies $\pi(s) <\pi(t)$.
Thus, condition (P2) holds. Also, it is easy to verify that
$\prer$ satisfies (P3).

If $x\in \ap$ denote its ``twin'' in $\aq $ by $x'$, and vice
versa, if $x\in \aq$ denote its ``twin'' in $\ap $ by $x'$.

Extend the definition  of $g$ as follows: $g:\ar\to \ap$ is a
function,
\begin{displaymath}
g(x)=\left\{ 
\begin{array}{ll}
x&\text{if $x\in \ap$,}\\  
x'&\text{if $x\in \aq$,}\\  
s&\text{if $x=y_s$ for some $s\in \ap$,}\\  
t&\text{if $x=u_t$ for some $t\in \ap$,}\\  
t'&\text{if $x=u_t$ for some $t\in \aq$.}  
\end{array}
\right .  
\end{displaymath}

For $\{s,t\}\in \br \ar;2;$ we will be able to define the infimum of
$s,t$ in $(\ar,\prer)$ from the infimum of $g(s),g(t)$ in $p$.
Now, we need to prove some facts concerning the behavior of the
function $g$ on $\ar$.

\begin{cclaim}
\label{Claim-1} 
Let $a\in A$ and $x\in \ar$. Then 
\begin{enumerate}[(1)]
\item $x\prer a$ iff $g(x)\prep a$,
\item $a\prer x$ iff $a\prep g(x)$.  
\end{enumerate}
\end{cclaim}

\begin{proof}
(1) $x\prer a$ iff $x\prepq a$ or $x\pren {1p} a$
and (1)
holds in both cases. \\  
(2) $a\prer x$ iff $a\prepq x$ or $a\pren {3p} x$ or 
$a \pren {4p} x$ or $a\pren {4q} x$,
and (2) holds in every case.  
\end{proof}

\begin{cclaim}
\label{Claim-2}  
If $x\prer y$ then $g(x)\prep g(y)$ for  $x,y\in \ar$.
  \end{cclaim}

\begin{proof}
$x\prenr y$   
iff $x\prenpq y$ or $x\pren {1p} y$ or $x\pren {1q} y$
or $x\pren 2 y$  or $x\pren {3p} y$ or
$x\pren {3q} y$ or $x \pren {4p} y$ or $x\pren {4q} y$
or $x \pren {5p} y$ or $x \pren {5q} y$, and the implication
holds in every case.
\end{proof}
\begin{cclaim}
\label{Claim-3} If $v\prep g(s)$ then 
$y_v\prer s$ for $v\in \ap\setm A$ and $s\in \ar$.
\end{cclaim}

\begin{proof}
If $s\in \ap$ ($s\in \aq$)  then $g(s)=s$  ($g(s)=s'$) and so 
$y_v \pren {1p} s$  ($y_v \pren {1q} s$).

If $s=y_x$ for some $x\in A_p$ then $g(s)=x$ and so 
$y_v \pre 2 y_x$. 

If $s=u_x$ for some $x\in L^+$
then $y_v \prer y_x$, and so $y_v \pren {4p} u_x$.
\end{proof}

\begin{cclaim}
\label{Claim-4}  If $x\prer y$ and
${\delta}_{g(x)}<{\delta}_{g(y)}$ then there is $a\in A$ such that 
$x\prer a \prer y$. 
\end{cclaim}

\begin{proof}
By  Claim \ref{Claim-2} we have $g(x)\prep g(y)$. Hence,
by  Claim \ref{Claim-2.11}, there is $a\in A$ such that $g(x)\prep a\prep g(y)$.
Then, by  Claim \ref{Claim-1}, we have $x\prer a\prer y$.  
\end{proof}

\begin{cclaim}
\label{Claim-5} If $a\in A$ and $x\in \ar$,
  $a\prer x$, then $\pin(a)\in \orb x$ iff $\pin(a)\in \orb{g(x)}$.
\end{cclaim}

\begin{proof}
We can assume that  $x\notin A_p\cup A_q $. If $x\in Y$ then Claim
\ref{Claim-2.15} implies the statement. If $x=u_z$ for some $z\in L^+$  then
$g(x)=z$, $\pin(a)<{\delta}_z$ and $\orb z\cap {\delta}_z =\orb
{u_z}\cap {\delta}_z = o_B(\pi_B(z))$, and so we are done.
\end{proof}

\begin{cclaim}
\label{Claim-6}  If $x\in \ar\setm A$, $v\in
  \ap\setm A$,
$v\prec_p g(x)$ and ${\delta}_v={\delta}_{g(x)}$ then $\pin(y_v)\in \orb x$.
\end{cclaim}

\begin{proof}  
We have $\pin(y_v)= \beta_v\in \eorb{{\climit}_v}\cap
[\gd({\delta_v}),\gu(\delta_v))$. If $x\in (\ap\cup
\aq)\setm A$, then $\beta_v\in \orb x$ by Claim \ref{Claim-2.14}.

 If $x=y_z$ for some $z\in A_p$, we have $z=g(x)$ and then
 ${\beta}_v\in \orb {y_z}$ by  Claim \ref{Claim-2.15}.

 If $x=u_z$ for some $z\in L^+$ then ${\beta}_v \in \orb z$
 because $p$ is good. Now as ${\beta}_v<{\delta}_z$ and
 $\orb z\cap {\delta}_z=\orb {u_z}\cap {\delta}_z$, the statement
 holds.
\end{proof}

\begin{cclaim}
\label{Claim-8}
If $s\in \ar\setm (A\cup Y)$ and $v=g(s)$ then $\pin(y_v)\in \orb s$.
\end{cclaim}

\begin{proof}
We have
$\pin(y_v)={\beta}_v\in  
\eorb{{\delta}_v}\cap \gu({\delta}_v) $.
If $s\in \ap\cup \aq$ then $\eorb{{\delta}_v}\cap \gu({\delta}_v)\subs \orb s$ because $p$
and $q$ are good.
If $s=u_{g(s)}$ then the block orbit of $s$ and the block orbit of
$g(s)$ are the same and the block orbit of $g(s)$ contains 
$\eorb{{\delta}_v}\cap \gu({\delta}_v)$ because $p$ is good.  
\end{proof}

\begin{cclaim}
\label{Claim-9} 
If $w\in \ap$, $s\in \ar$, $w\prer s $ and ${\delta}_w={\delta}_{g(s)}$ 
then $s\in \ap$.
\end{cclaim}

\begin{proof}
If $s\in \aq\setm \ap$ then $w\prepq s$ and so there is $a\in A$
such that $w\prep a\preq s$ which contradicts ${\delta}_w={\delta}_{g(s)}$.  

If $s=y_{g(s)}$ then $w\pren {3p} s$, i.e. there is $a\in A$
with $w\prep a\prep g(s)$ which contradicts
${\delta}_w={\delta}_{g(s)}$.

If $s=u_{g(s)}$ then $w \pren {4p} u_{g(s)}$, i.e. 
$w\prer y_{g(s)}$, but this was excluded in the previous paragraph.
 \end{proof}

\begin{lemma}
\label{Lemma 2.21}  There is a function
$\ir\supset \ip\cup \iq$ such that  $\<\ar,\prer,\ir\>$ satisfies
(P4) and (P5).
\end{lemma}

\newcases
\begin{proof}

\mbox{ }
If $\{s,t\}\in \br \ap;2;$  ($\{s,t\}\in \br \aq;2;$)
we will have $\ir\{s,t\}=\ip\{s,t\}$  ($\ir\{s,t\}=\iq\{s,t\}$), and so 
 (P5) holds because $p$ and $q$ satisfy (P5).

To check (P4) we should prove that 
 $\ip\{s,t\}$ is the greatest common lower bound of 
$s$ and $t$ in $(\ar,\prer)$. 

Indeed, let
$x\prer s,t$. We can assume that $x\notin \ap$.
Then, we distinguish the following
three cases. 
\begin{acase}
$x\in \aq\setm \ap$.
\end{acase}
Then there are $a,b\in A$ such that $x\preq a \prep s$ and $x\preq
b\prep t$. Thus $x\preq \iq\{a,b\}=\ip\{a,b\}\prep \ip\{s,t\}$ and
so $x \prepq \ip\{s,t\}$.
\begin{acase}
$x\in Y$.
\end{acase}
Then $x\pren {1p}s$ and $x\pren {1p}t$ , i.e. $g(x)\prep s$ and
$g(x)\prep t$. So $g(x)\prep \ip\{s,t\}$ and hence $x\pren
{1p}\ip\{s,t\}$.
\begin{acase}
$x\in U$.
\end{acase}
Put $x=u_z$ for some $z\in L^+$.
Since $x\prer s,t$, we have that $u_z\pren{5p} s,t$, and thus
$z\; \mip\; s,t$. So $z\; \mip\; i_p\{s,t\}$, and hence $x\;\mir\;
i_p\{s,t\}$.

Assume now that 
$s,t\in \ar$ are $\prer$-compatible, but $\prer$-incomparable
elements, $\{s,t\}\notin \br \ap;2;\cup \br \aq;2;$. 
Write $v=\ip\{g(s),g(t)\}$.  
 Note that, by  Claim \ref{Claim-2}, $g(s)$ and $g(t)$ are compatible in $p$
 and hence $v\in A_p$.
Let
\begin{displaymath}
\ir\{s,t\}=\left\{ 
\begin{array}{ll}
v&\text{if $v\in A$,}\\  
y_v&\text{otherwise.}  
 \end{array}
\right .  
\end{displaymath}

\renewcommand{\theccase}{\Roman{ccase}}
\renewcommand{\thesscase}{\thescase.\roman{sscase}}

\begin{ccase}
$v\in A$.  
\end{ccase}

Then $g(s)$ and $g(t)$ are incomparable in $\ap$.
Indeed, $g(s)\prep g(t)$ implies $v=g(s)$ and so 
$s=g(s)\prer t$ by  Claim \ref{Claim-1}.

 Thus $\pin(v)\in \orb{g(s)}\cap \orb{g(t)}$ by applying (P5) in $p$.
 Note that $v\prer s,t$ by  Claim \ref{Claim-1}. So, $\pin(v)\in \orb{s}\cap \orb{t}$
by  Claim \ref{Claim-5}. Hence (P5) holds.

We have to check that $v$ is
the greatest lower bound of $s,t$ in $(A_r,\prer)$.
We have $v\prer s,t$ by  Claim \ref{Claim-1}.

Let $w\in \ar$, $w\prer s,t$. Then 
$g(w)\prep g(s),g(t)$ by  Claim \ref{Claim-2}. So 
$g(w)\prep v$. Then $w\prer v$ by  Claim \ref{Claim-1}. 

\begin{ccase}
$v\notin A$.  
\end{ccase}

Then ${\delta}_{g(s)}={\delta}_{g(t)}={\delta}_v$
by  Claim \ref{Claim-1} and  Claim \ref{Claim-2.11} if $g(s)$ and $g(t)$ are comparable in $\ap$,
and by  Claim \ref{Claim-2.13} if $g(s)$ and $g(t)$ are incomparable in $\ap$.

If $g(s)$ and $g(t)$ are incomparable in $A_p$ then $v \prec_p
g(s),g(t)$ and $s,t\not\in A$ by  Claim \ref{Claim-2.12}. So, $\pin(y_v)\in
\orb s\cap \orb t$ by Claim \ref{Claim-6}.

If $g(s)\prec_p g(t)$  then $s\notin Y$ by  Claim \ref{Claim-3}
     and $s\not\in A$ because $v=g(s)\not\in A$.
Then $\pin(y_v)\in \orb{s}$ by  Claim \ref{Claim-8}.
Also, since $v=g(s) \prec_p g(t)$ we infer from Claim \ref{Claim-1} that
$t\not\in A$ and so we have that $\pin(y_v)\in \orb t$ by  Claim \ref{Claim-6}.
Hence (P5) holds.

We have to check that $y_v$ is
the greatest common lower bound of $s,t$ in $(A_r,\prer)$.
First observe that $y_v\prer s,t$ by  Claim \ref{Claim-3}. 

Let $w\prer s,t $.

Assume first that 
${\delta}_{g(w)}<{\delta}_v$.  
Then there are $a,b\in A$ with $w\prer a\prer s$ and 
$w\prer b\prer t$ by  Claim \ref{Claim-4} and so 
$g(w) \prep \ip\{a,b\} \prep v$ by using Claim \ref{Claim-1}.
Now since $g(y_v)=v$, we obtain  $w\prer \ip\{a,b\}\prer y_v$ again 
by  Claim \ref{Claim-1}.

Assume now that 
${\delta}_{g(w)}={\delta}_v$.  
 Since $\{s,t\}\not\in [A_p]^2\cup [A_q]^2$, we have that
 $w\not\in U\cup U'$.
Then, by  Claim \ref{Claim-9},  
$w=y_z$ for some $z\in \ap$. Then $z\prep g(s)$ and $z\prep g(t)$
by Claim \ref{Claim-2},
and so $z\prep v$. Thus $y_z\prer y_v$.
\end{proof}

 Now our aim is to verify condition (P6). First, we need some
 preparations.

For every  $x,y\in \ar$ with  $x\prer y$
let 
\begin{displaymath}
\pi_x(y)=\left\{ 
\begin{array}{ll}
\pin(y)&\text{if $\pib(x)=\pib(y)$,}\\  
\pim(y)&\text{if $\pib(x)\ne\pib(y)$.}  
\end{array}
\right .  
\end{displaymath}

  Note that for every $x,y\in\ar$ with $x\prer y$,
an interval  $\Lambda\in \mathbb I$ isolates  $x$ from $y$ iff
 $\bott \Lambda <\piz (x)<\topp
\Lambda\le \pi_x (y) $.

\begin{cclaim}
\label{Claim-10}
Let $a\in A$ and $t\in \ar$, $a\prer t$. 
If $\Lambda$ isolates $a$ from $t$
then $\Lambda$ isolates $a$ from $g(t)$.
\end{cclaim}

\begin{proof}
    The statement is obvious if $t\in A_p$. Assume that $t\in
    A_q\setminus A_p$. Note that since $\Lambda$ contains an element
    of $A$, we have that $\Lambda^+\in Z$.  Now if $t\in D\cup F\cup
    M$  we have that $Z\cap \pin(t) = Z\cap \pin(g(t)) = Z\cap
    \gamma(\delta_t)$, and so we are done. If $t\in L$  then as
    $a\prer t$ we infer that
    $\pi_B(a)\neq \pi_B(t)$ and $\pin(a) < \delta_t = \pi_{-}(t)$,
    hence we have $\pin(a) < \Lambda^+ \leq \pi_a(t) = \pi_a(g(t))
    = \pi_{-}(t)$, and so the statement holds.

If $t=y_v$ for some $v\in A_p$, then $a \prec_p v=g(t)$ and
$\pi_a(y_v)\le \pi_a(v)$, and so we are done.

If $t= u_v$ for some $v\in L^+$, we have $a \prec_p v=g(t)$ and
$\pi_a(u_v)= \pi_a(v)=\pim(v)$.
\end{proof}

\begin{cclaim}
\label{Claim-11} 
Let $a\in A$ and $x\in \ar\setm (\ap\cup \aq)$, $x\prer a$. 
If $\Lambda$ isolates $x$ from  $a$
then $x=y_{g(x)}$ and $\Lambda$ isolates $g(x)$ from  $a$.
\end{cclaim}

\begin{proof}
    We have $g(x)\prep a$ by  Claim \ref{Claim-1}, so as $a\in A$ we infer
    that $g(x)\not\in L\cup M$, and thus $x\not\in U\cup U'$.
    Hence $x\in Y$ and $g(x)\in D\cup F$, and so $x=y_{g(x)}$ and
    $\pin(g(x)) < \delta_{g(x)}$.

Let $J({\delta}_{g(x)})=\contint {\pin(g(x))}{j}$
and $\Lambda=\contint {\pin(x)}{\ell}$.
If $\ell>j$ then $\bott\Lambda=\pin({y_{g(x)}}) = \pin(x)$,
which is impossible.
If $\ell\le j$ then $J({\delta}_{g(x)})\subs \Lambda$
and so $\bott{\Lambda}<\pin(g(x))<\topp \Lambda$, i.e.  
$\Lambda$ isolates $g(x)$ from  $a$.  
\end{proof}

\begin{lemma}
\label{Lemma 2.22}  $(A_r,\prer,i_r)$
satisfies $(P6)$.
\end{lemma}

\begin{proof}
Assume that $\{s,t\}\in \br \ar;2;$, $s\prer t$ and $\Lambda$
isolates ${s}$ from ${t}$. Suppose that $\pin(s)\neq\pi_{-}(s)$ if
$s\not\in B_S$. So, $s\not\in U\cup U'$. We should find $v\in A_r$
such that $s\;\mir\; v\;\mir\; t$ and $\pin(v)=\Lambda^+$. 
Note that since $s\prer t$, we have $\delta_{g(s)}\leq
\delta_{g(t)}$ by Claims  \ref{Claim-2} and \ref{Claim-2.10}.

We can assume that 
  $\{s,t\}\notin \br \ap;2;\cup \br \aq;2;$
because $p$ and $q$ satisfy (P6).

\newcases
\begin{ccase}
${\delta}_{g(s)}<{\delta}_{g(t)}$.  
\end{ccase}

By  Claim \ref{Claim-4} there is $a\in A$ with $s\prer a \prer t$.
Moreover, $g(s)\prep a \prep g(t)$ by  Claim \ref{Claim-1}.

\begin{scase}
$\pin(a)\in \Lambda$.  
\end{scase}

Then $\pib(s)=\pib(a)$ and so $\pi_s(t)=\pi_a(t)$.
Thus $\Lambda$ isolates $a$ from  $t$.

If $t\in \ap $ $(t\in \aq)$ then applying (P6)  in $p$  (in $q$)
for $a$, $t$ and $\Lambda$ we obtain $b\in \ap$  ($b\in \aq$) 
such that $a\prep b\prep t$  ($a\preq b\preq t$) and $
\pin(b)=\topp{\Lambda}$. Then $s\prer a \prepq b\prepq t$, so we are done.

Assume now that $t\notin \ap\cup \aq$.

By  Claim \ref{Claim-10}, the interval  $\Lambda$ isolates $a$ from  $g(t)$ .
Since $\pim(a)\ne \pin(a)$ 
if $a\not\in B_S$,
we can apply (P6) in $p$ to get a
$b\in \ap $ with $\pin(b)=\topp{\Lambda}$ and $a\prep b\prep g(t)$.

       Note that as $\pin(a)\in \Lambda, a\in A$ and $\pin(b)
       =\Lambda^+$, we have that $\pin(b)\in Z$.

If  $\pib(a)=\pib(b)$,  we have $b\notin M\cup L$ because $a\in
A$.

If $\pib(a)\neq \pib(b)$, then $\pim(b)=\pin(b)=\Lambda^+\leq
\pin(t)$. Note that if $t\in U\cup U'$, then $\pi(t)=\Lambda^+$,
and so we are done. Thus, we may assume that $t\in Y$. Then, we
have $\pib(b)=\pib(t)=\pib(g(t))$ and $g(t)\in F$. Hence $b\in
K\cup F$.

In both cases we have $b\notin M\cup L$, so
 ${\pin}(b)\in Z$ implies $b\in A$.
Thus $b\prer t$ by  Claim \ref{Claim-1}, and so $b$ witnesses (P6).

\begin{scase}
$\pin(a)\notin \Lambda$.  
\end{scase}

Since $p$ and $q$ satisfy $(P6)$ and $\Lambda$ isolates $s$ from $a$, 
we can assume that $s\notin \ap\cup \aq$.

Hence $s=y_{g(s)}$ and 
$\Lambda$ isolates $g(s)$ from $a$ by  Claim \ref{Claim-11}. Since 
$\pin(g(s))\ne \pim(g(s))$ if $g(s)\not\in B_S$,
 there is 
$v\in \ap$ with $g(s)\prep v\prep a$ and  $\pin(v)=\topp {\Lambda}$.
Since $y_{g(s)}\prer g(s)$ by the definition of
     $\prer$, we have that $v$ witnesses (P6).

\begin{ccase}
${\delta}_{g(s)}={\delta}_{g(t)}$.  
\end{ccase}
 
\begin{scase}
$s\in \ap$.  
\end{scase}

     Since $s\in A_p$, $s\prer t$ and $\delta_s = \delta_{g(t)}$
     we infer from  Claim \ref{Claim-9} that $t\in A_p$, which was excluded.

     By means of a similar argument, we can show that $s\in A_q$
     is also impossible.

\begin{scase}
$s=y_v$ for some $v\in \ap$.  
\end{scase}

         We have that $\delta_v = \delta_{g(t)}$. Note that since
         $\Lambda^- <\pin(s) < \Lambda^+$, we have
         $\delta_v\leq \Lambda^+$.

Thus  $\pin(t)\ge \topp{\Lambda}\ge {\delta}_{v}={\delta}_{g(t)}$.
Since we can assume that $\pin(t)>\topp{\Lambda}$,
we have $\pin(t)>{\delta}_{g(t)}$. 
If $t\in \ap\cup\aq$ and $g(t)\in F\cup D\cup M$, or 
$t\in Y$, or $t\in U\cup U'$ then 
$\pin(t)\le {\delta}_{g(t)}$. Thus we have $t\in \ap\cup \aq$
and  $g(t)\in L$.

 Note that as $\pi_B(t)\neq S$, if $\pib(y_v)=\pib(t)$ we would
 infer that $v\in F$ and hence $\delta_t = \delta_{g(t)} <
 \delta_v$. So $\pib(s)\neq \pib(t)$. Now since $\Lambda$
 isolates $s$ from $t$, we deduce that $\delta_v = \delta_t =
 \Lambda^+$, and hence $\Lambda = \jint{\delta_t}$.

Assume that $t\in \aq$ (the case $t\in \ap$ is simpler).
Then $g(t)=t'\in L$.
Since  $\pin(t)>{\delta}_{t}=\pim(t)$ we have $\pin(t')>\pim(t')$ and so $t'\in L^+$.

Since $y_v\prer t$ we have  $y_v\pren{1q} t$, i.e.   
$v\prep t'$ and so $y_v\pre{2} y_{t'}$.
Thus $y_v\pren{4q} u_t$.
Hence $y_v\prer u_t\prer t$ and $\pin(u_t)={\delta}_t=\topp \Lambda$,
i.e. $u_t$ witnesses that (P6) holds.
\end{proof}

This completes the proof of Lemma \ref{Lemma-2.3}, i.e. $\pcal$ satisfies
${\kappa}^+$-c.c.
\end{proof}


\end{document}